\documentclass{amsart}

\usepackage{amsmath, amsthm, amssymb, tikz, xypic}

\newtheorem{thm}{Theorem}[section]

\newtheorem{lem}[thm]{Lemma}
\newtheorem{cor}[thm]{Corollary}
\newtheorem{prop}[thm]{Proposition}

\theoremstyle{definition}
\newtheorem{dfn}[thm]{Definition}
\newtheorem{prb}[thm]{Problem}
\theoremstyle{remark}

\newcommand{\im}{\operatorname{im}}

\newcommand{\C}{\mathbb{C}}
\newcommand{\R}{\mathbb{R}}
\newcommand{\N}{\mathbb{N}}
\newcommand{\Z}{\mathbb{Z}}
\newcommand{\Q}{\mathbb{Q}}

\newcommand{\bb}[1]{\mathbb{{#1}}}
\newcommand{\inv}{^{-1}}
\newcommand{\ip}[1]{\langle {#1} \rangle} 
\newcommand{\sm}{\smallsetminus}
 
\newcommand{\Fr}{\mathrm{Fr}}

\newcommand{\tr}{\operatorname{tr}}

\newcommand{\aut}{\operatorname{Aut}}

\newcommand{\res}{\upharpoonright}

\newcommand{\m}{\operatorname{\mu}}

\newcommand{\lra}{\Leftrightarrow}

\newcommand{\s}[1]{\mathcal{{#1}}} 

\newcommand{\cd}{\mathrm{CD}}

\newcommand{\fiid}{fiid }

\renewcommand{\bf}{\mathbf}
\newcommand{\ax}{\mathsf{a}}
\newcommand{\bx}{\mathsf{b}}

\newcommand{\cay}{\operatorname{Cay}}
\newcommand{\sch}{\operatorname{Sch}}
\newcommand{\fix}{\operatorname{fix}}

\title{Factor maps for automorphism groups via Cayley diagrams}
\author{Riley Thornton}
\begin{document}

\begin{abstract}

We leverage a correspondence between group actions and edge-labelled graphs in two ways. First, we give a unified presentation of several folklore results connecting weak containment, local-global convergence, and continuous model theory. Second, we investigate the difference between $\aut(\cay(\Gamma))$-fiid combinatorics and $\Gamma$-fiid combinatorics for various marked groups $\Gamma$. It's straightforward to see that these differences vanish when $\cay(\Gamma)$ admits an $\aut(\cay(\Gamma))$-fiid Cayley diagram. We extend this to show that the approximate combinatorics are the same when $\cay(\Gamma)$ admits an approximate fiid Cayley diagram, and we give several examples and nonexamples of groups whose Cayley graphs admit (approximate) fiid Cayley diagrams. 

In particular, we show that trees admit approximate Cayley diagrams for any group whose Cayley graph is a tree; Cayley graphs of torsion free nilpotent groups do not admit fiid Cayley diagrams; and there are groups with isomorphic Cayley graphs so that only one them admits even an approximate Cayley diagram (in fact our construction answers a question of Weilacher).

\end{abstract}

\maketitle

\section{Introduction}

The basic idea of this paper is that actions of marked groups correspond exactly to graphs with edge labellings (satisfying certain local combinatorial constraints). For instance, an action of $C_2^{*n}$ is the same as an edge $n$-colored graph: Given an action we can form the Schreier graph and color edges by the generators that induce them (a generator $\gamma$ may have some fixed points, which won't see an edge of color $\gamma$), and given a graph with an edge coloring we can let each generator act by swapping the endpoints of edges with the corresponding color.

We put this idea to work in two ways. We give a unified exposition of the connections between three closely related notions from ergodic theory, combinatorics, and model theory. And, we investigate when $\Gamma$-fiid processes can be lifted to $\aut\bigl(\cay(\Gamma)\bigr)$-fiid processes for various marked groups $\Gamma$.

We are generally interested in finding measurable solutions to combinatorial problems on pmp graphs (elsewhere called graphings). For instance, if I have a pmp action of $F_2$ how large of an independent set can find in its Schreier graph? What about an arbitrary 4-regular treeing? Or, can we find a 2-regular spanning tree in each component of the Schreier graph of a $\Z^2$ action? What about in any graphing whose components are all square grids? 

The three notions we explore in the first section of this paper all try to capture something about the local statistics that can be approximated by labellings of a pmp graph. Namely, we will explain the exact relation between weak containment from dynamics, local-global convergence from combinatorics, and existential theories from continuous model theory. 

In the second section we investigate when $\Gamma$-factor of iid labellings of $\cay(\Gamma)$ (or equivalently labelling of the $\Gamma$-Bernoulli shift) can be transferred to $\aut\bigl(\cay(\Gamma)\bigr)$-fiid labellings (or equivalently, labellings of the Bernoulli graphing over $\cay(\Gamma)$.) For example, most known constructions of $F_n$-fiid processes in fact yield labellings for any $2n$-regular treeing. Our results explain this phenomenon to some extent, and they allow analysis of $\aut(T_{2n})$-fiid maps to rule out some $F_n$-fiid processes.

\subsection{Statement of results}

Folklore has it that weak-containment, local-global convergence, and $\Sigma_1$-equivalence are all (in some sense) different languages for capturing the same phenomena. In the fist section of this paper, we'll explain the equivalences in the detail and record proofs. The statements are the cleanest for free actions:

\begin{thm}
    For free pmp actions $\ax,\bx:\Gamma\curvearrowright (X,\mu)$, the following are equivalent:
\begin{enumerate}
    \item $\ax$ weakly contains $\bx$
    \item The local statistics of $\bx$-labellings are all approximated by local statistics of $\ax$-labellings
    \item The $\Sigma_1$ theory of $M(\bx)$ is dominates that of $M(\ax)$, where $M(\ax)$ is $L^2(X,\mu)$ equipped with the automorphisms coming from $\ax$.
\end{enumerate}
\end{thm}

Here, we say $\ax$ weakly contains $\bx$ to mean some ultrapower of $\ax$ factors onto $\bx$. And, $M(\ax)$ is the continuous model associated to $\ax$, i.e.~$L^2(X,\mu)$ as a commutative $*$-algebra on a Hilbert space equipped with automorphisms induced by $\ax$. There are similar theorems for the associated equivalence and convergence notions. 

For nonfree actions, local-global equivalence includes information about the ismorphism types of components in the Schreier graph, which complicates matters:

\begin{thm}
    For a pmp actions, $\ax,\bx:\Gamma\curvearrowright (X,\mu)$ the following are equivalent
\begin{enumerate}
    \item The labelled Schreier graph $\ax$ local-global contains that of $\bx$
    \item The $\Sigma_1$-theory of $M^+(\ax)$ is dominated by that of $M^+(\bx)$, where $M^+(\ax)$ is $M(\ax)$ equipped with constants for the set of fixed points of each group element
    \item $\ax$ weakly-contains $\bx$ and the factor map from the ultrapower preserve stabilizers.
\end{enumerate}
\end{thm}

These subtleties all disappear for the equivalence and convergence notions, though. Lastly, we can extend these equivalences to pmp graphs by noting that every graph of degree at most $\Delta$ is induced by an action of $C_2^{*(2\Delta+1)}$. We call such an action a marking.

\begin{thm}
For pmp graphs $G$, $H$, the following are equivalent
    \begin{enumerate}
        \item $G$ local-global contains $H$
        \item For any marking $\bx$ of $H$, there is $G'$ local-global equivalent to $G$ and a marking $\ax$ of $G'$ which weakly contains $\bx$
        \item There is some marking $\bx$ of $H$ and some $G'$ which is local-global equivalent to $G$ and has a marking which weakly contains $\bx$.
    \end{enumerate}
\end{thm}

With these theorems in hand, tools like ultraproducts and the L\"owenheim--Skolem theorem give immediate proofs of various compactness, continuity, and representation theorems for the space of local-global convergence.

In the second section of this paper, we will try to understand to what extent there are differences in the measurable combinatorics of free $\Gamma$ actions and the measurable combinatorics of graphs locally isomorphic to $\cay(\Gamma)$. Of course, this depends on the (marked) group $\Gamma$.

We will focus on the simplest case of a $\Gamma$ action-- the Bernoulli shift. The corresponding construction for pmp graphs is the Bernoulli graphing over $\cay(\Gamma)$. Measurable labellings of the Bernoulli shift and the Bernoulli graphing correspond to $\Gamma$-fiid and $\aut(G)$-fiid processes respectively. (This nomenclature is unfortunate in this context, but somewhat standard).

Ostensibly finding an $\aut(G)$-fiid labelling of $G$ is stronger than finding a $\Gamma$-fiid labelling. For instance, Rahman and Virag give asymptotically optimal upper bounds on $\aut(T_n)$-fiid independent sets \cite{RahmanVirag}. This doesn't \emph{a priori} give us upper bounds on the density of $F_n$-fiid independent sets, yet every nontrivial construction we know of for finding $F_n$-fiid labellings (such as the Lyons--Nazarov matching theorem and the matching lower bound for independent sets) in fact gives $\aut(T_{2n})$-fiid labellings. We want to know: Is there separation? Or are there transfer theorems?

We show that these questions boil down to whether one can find an $\aut\bigl(\cay(\Gamma)\bigr)$-fiid Cayley diagram.

\begin{thm}
    Consider a marked group $\Gamma$ with Cayley graph $G$. The following are equivalent:
    \begin{enumerate}
        \item For any $\Gamma$-fiid labelling of $G$ there is an $\aut(G)$-fiid labelling with the same local statistics
        \item $G$ admits an $\aut(G)$-fiid $\Gamma$-Cayley diagram
    \end{enumerate}
    And, the following are equivalent:
    \begin{enumerate}
        \item[(1*)] For any $\Gamma$-fiid labelling of $G$ of $f$, the local statistics of $f$ are approximated by statistics of $\aut(G)$-fiid labellings 
        \item[(2*)] $G$ admits an $\aut(G)$-fiid approximate Cayley diagram.
    \end{enumerate}
\end{thm}

And we prove the existence and non-existence of Cayley diagrams for a number of groups. For instance:

\begin{thm}
    For any torsion free nilpotent group $\Gamma$, $\cay(\Gamma)$ does not admit an $\aut\bigl(\cay(\Gamma)\bigr)$-fiid Cayley diagram.
\end{thm}

\begin{thm}
    The $n$-regular tree, $T_n$, admits an approximate Cayley diagram for any group whose Cayley graph is $T_n$.
\end{thm}

As a corollary, the $\aut(T_{2n})$- and $F_n$-fiid independence numbers of trees are the same. So, the Rahman--Virag bound applies to $F_n$ as well.

\subsection{Notation and conventions}

A graph on a vertex set $V$ is a symmetric irreflexive subset of $V^2$. We do not allow graphs to have loops or parallel edges. So, the edges in a graph are ordered pairs, and if $(x,y)$ is and edge so is $(y,x)$.

A marked group is a group with a specified symmetric generating set. We usually leave the generating set implicit in our notation. If $\Gamma$ is a marked group with a symmetric generating set $E$, and $\ax:\Gamma\curvearrowright X$ is any action of $\Gamma$, we can form the associated Schreier graph on $X$:
\[\sch(\ax):=(X,\{(x,\gamma\cdot x): x\in X, \gamma\in E, \gamma\cdot x\not=x\}).\] Note that for this definition to give us a graph, we need $E$ to be symmetric and we need to ignore fixed points. 

The Cayley graph of $\Gamma$, $\cay(\Gamma)$, is the Schreier graph of $\Gamma$ acting on itself by multiplication. So, any free action of $\Gamma$ has a Schreier graph where every component is isomorphic to $\cay(\Gamma)$. There is one ambiguity in this definition: whether $\Gamma\curvearrowright\Gamma$ is by left or right multiplication. Of course it doesn't matter in any deep way, but we need to set a convention. We'll choose left multiplication. So,
\[\cay(\Gamma)=\{\Gamma, \{(x,\gamma x): x\in \Gamma, \gamma\in E\}.\]

For a permutation group $\bb G\leq \{f\in A^A: f\mbox{ is a permutation}\}$, and any set $X$, the shift action of $\bb G$ on $X^A$ is
\[\mathsf{s}: \bb G\curvearrowright X^A\]
\[\gamma\cdot f=f\circ \gamma\inv.\] Typically, $\bb G$ will be an automorphism group of some countable graph. In case $\bb G$ is the automorphism group of a Cayley graph, $\Gamma$ acts on $\cay(\Gamma)$ by \emph{right} multiplication. So, we have an embedding
\[\ip{\cdot}: \Gamma\hookrightarrow \bb G\]
\[\ip{\gamma}(x)=x\gamma\inv.\] And, the shift action is given by
\[\mathsf{s}: \Gamma\curvearrowright X^A\]
\[(\gamma \cdot f)(x)=(f\circ \ip{\gamma}\inv)(x)=f(x\gamma).\]

We will construct a number of exotic measure spaces throughout this paper, but we are interested mostly in their standard quotients. By a standard probability space, we mean the completion of a regular Borel probability measure on completely metrizable space. There several equivalent conditions, but all standard spaces isomorphic to an interval with Lebesgue measure (possibly along with a few atoms). see \cite{msrthry} and for stronger results in the Borel setting \cite{Kechris}.

For a set $A\subseteq X$, we write $\bf 1_A$ for the characteristic function of $A$:
\[\bf 1_A(x)=\left\{\begin{array}{cc} 1 & x\in A\\
 0 & else\end{array}\right. .\] Aside from characteristic functions, we typically reserve bold characters for random variables. Typically the domains of our random variables will be product spaces of the form $([0,1]^V, \lambda^V)$ where $\lambda$ is Lebesgue measure and $G=(V, E)$ is a countable graph (or a quotient of such a space). Of course $([0,1],\lambda)$ is isomorphic to any standard probability space, such as $([0,1]^\N,\lambda^\N)$. So, in practice, we assume every vertex of our graph can independently sample any source of randomness we require infinitely often.

Usually, our random variables will take values in $A^V$ for some $A$ and will be $\Gamma$-equivariant for some $\Gamma\leq \aut(G)$; this is what is meant by an fiid labelling of $G$. More generally, $\bf f$ is $\Gamma$-fiid if $\Gamma$ acts on both $V$ (and by extension $[0,1]^V$) and the codomain of $\bf f$, and for all $\gamma\in \Gamma$ \[\bf f(\gamma\cdot x)=\gamma \cdot \bf f(x).\] We suppress the dependent variable in our notation. So, we write $\bb P\bigl(\phi(\bf f)\bigr)$ for the probability that $\bf f$ has property $\phi$, e.g. for a vertex $v\in V$
\[\bb P\bigl(\bf f(v)=3\bigr)=\lambda^V\bigl(\bigl\{x:\bigl(\bf f(x)\bigr)(v)=3\bigr\}\bigr)\]

Similarly $\bb E\bigl(g(\bf f)\bigr)$ is the expected value of $g(\bf f)$, e.g.
\[\bb E\bigl(\# \bigl\{ u\in B_3(v):\bf  f(u)=3\bigr\}\bigr)=\int_x \# \bigl\{ u\in B_3(v): \bigl(\bf f(x)\bigr)(u)=3\bigr\}.\]

\subsection{Acknowledgements} Thanks to Andrew Marks, Clinton Conley, Eric Wang, and Laszlo T\'oth for helpful conversations and comments. This work was partially supported by NSF MSPRF grant DMS-2202827.

\section{Weak containment and related notions}

An ultrapower of a pmp group action is a kind of compactification of the action with respect to labellings or maps out of the space. For instance, if an action admits $\epsilon$-invariant sets for all $\epsilon$, then we can take a limit of their characteristic functions to find the characteristic function of an invariant set in the ultrapower. Ultrapowers are studied extensively in, for example, \cite{CKTD}. One can take as definition the point realization of the ultrapower $M(\ax)^{\s U}$ given in Definition \ref{dfn: ultraproducts}.

\begin{dfn}\label{dfn: weak containment actions}
    A pmp action $\ax:\Gamma\curvearrowright (X,\mu)$ \textbf{contains} an action $\bx:\Gamma\curvearrowright Y$ if $\ax$ factors onto $\bx$ via a measure preserving equivariant map.

    An action $\ax$ \textbf{weakly contains} an action $\bx$ if some ultrapower of $\ax$ contains $\bx$. In this case, we write
    \[\bx\preccurlyeq_w \ax.\]
\end{dfn}

This is a somewhat nonstandard definition, but it is equivalent to any of the other definition you like. There are many; for a survey, see \cite{GlobalAspects}. In essence, approximate properties of an action are reflected in exact properties of its ultrapower, so weak containment can be read as saying that $\ax$ approximately factors onto $\bx$.

There are number of folklore results which connect weak containment to notions from model theory and combinatorics. In this section, we'll explain these connections in detail. We start by considering the case of free pmp actions of groups, then we explain the modifications and subtleties involved with non-free actions, and finally we translate the theory of non-free actions to a theory of pmp graphs.

\subsection{Model theory notions}

For the sake of notational convenience, I will assume $\Gamma=\ip{\gamma_1,...,\gamma_n}$ is finitely generated.

\begin{dfn}
    For a (possible nonstandard) measure preserving action $\ax:\Gamma\curvearrowright X$, the associated \textbf{continuous model} is the structure
    \[M(\ax)=\bigl(L^2(X,\mu), \ax(\gamma_1),...,\ax(\gamma_n)\bigr)\]
    where we abuse notation slightly and write $\ax(\gamma)$ for the automorphism of $L^2(X,\mu)$ given by the action: $\ax(\gamma)(f)=\gamma\cdot_{\ax} f,$ so $\bigl(\ax(\gamma)(f)\bigr)(x)=f(\gamma\inv\cdot_{\ax} x)$. 

    We give $L^2$ the usual metric
    \[d(f,g)^2=\int_X |f-g|^2\;d\mu.\] When we speak of separability of $M(\ax)$, we mean with respect to this metric. And, we equip $L^2(X,\mu)$ with pointwise multiplication and conjugation as well as the usual Hilbert space operations.

    A \textbf{map} from $M(\ax)$ to $M(\bx)$ is an equivariant linear map that preserves pointwise products, pointwise conjugates, and inner products.
\end{dfn}

Viewing $M(\ax)$ as an action of $\Gamma$ on a commutative tracial {von Neumann} algebra, we can invoke a generality duality theory. The essential point is that we can recover $\ax$ from $M(\ax)$ as follows:
\begin{prop}
    For any pmp action $\ax$, 
    \begin{enumerate}
        \item $f\in M(\ax)$ is a characteristic function if and only if $f^2=f=f^*$
        \item The measure of a set $A$ is $\tr(\bf 1_A)=\ip{\bf 1, \bf 1_A}$
        \item For any sets $A, B$, the Boolean operations become: $\bf 1_{A\cap B}=\bf 1_A\bf 1_B$ and $\bf 1_{A\cup B}=\bf 1_A+\bf 1_B-\bf 1_{A\cap B}$
        \item The action becomes: $\bf 1_{\gamma A}=\gamma\inv \bf 1_{A}. $
    \end{enumerate}
\end{prop}
More generally, given any unital commutative $*$-algebra on a Hilbert space, $\s A$, we can construct a Boolean algebra equipped with a premeasure by considering projections in the algebra as above. Automorphisms of $\s A$ translate into automorphisms of the Boolean algebra, And, if the algebra is separable and we only have finitely many automorphisms (or even countably many) automorphisms, we can always find a point realization, i.e.~$(\s A, f_1,..., f_n)\cong M(\bx)$ for some pmp action $\bx$. This duality lets us study $\ax$ by looking at algebraic properties of isometries of some metric space, so we can apply the tools of continuous model theory.

The following is a special case of a more general notion of ultraproducts for metric structures:

\begin{dfn}\label{dfn: ultraproducts}
    For a sequence of actions $\ax_i:\Gamma\curvearrowright (X_i,\mu_i)$ and an ultrafilter $\s U$ on $\N$, the \textbf{ultraproduct} of the $M(\ax_i)$s is:
    \[\lim_{\s U} M(\ax_i)=\bigl (\lim_{\s U} L^2(X_i,\mu_i), \lim_{\s U} \ax_i(\gamma_1),...,\lim_{\s U} \ax_i(\gamma_n)\bigr).\] Here $\lim_{\s U} L^2(X_i,\mu_i)$ is the ultraproduct as a tracial {von Neumann} algebra, i.e. the structure with domain \[\{\ip{x_i:i\in\N}: (\exists N)(\forall i)\; x_i\in M(\ax_i)\mbox{ and }|x_i|<N\}/\sim_{\s U}\] where \[ x \sim_{\s U} y\;:\lra \;\lim_{\s U} d(x_i,y_i)=0.\]
    
    We call the $\sim_{\s U}$ equivalence class $[x]_{\s U}$ the \textbf{ultralimit} of the sequence $x$. All operations are defined as pointwise ultralimits. For instance:
    \[[x]_{\s U}+[y]_{\s U}= [\ip{x_i+y_i:i\in\N}]_{\s U}\] 
    \[ \tr\bigr([x]_{\s U}\bigl)=\lim_{\s U}\tr(x_i)\]and \[\ax^{\s U}(\gamma)\bigl([x]_{\s U}\bigr)=[\ip{\ax_i(\gamma)(x_i):i\in\N}]_{\s U}.\]

    These limits all exists since we restrict to bounded sequences. We write $M(\ax)^{\s U}$ for the case where all the $\ax_i$ are identical to $\ax$. In this case, we call $M(\ax)^{\s U}$ the \textbf{ultrapower} of $M(\ax)$.
\end{dfn}

One can check that $\lim_{\s U} M(\ax_i)$ is a commutative tracial {von Neumann} algebra equipped with a set of automorphisms which generate an action of $\Gamma$ (i.e.~it has nice enough algebraic properties to recover a Boolean algebra and premeasure as sketched above). One way to check this is to use Theorem \ref{thm:los} below. The ultralimit is not separable (unless it's finite dimensional), but it turns out that $\lim_{\s U} M(\ax_i)$ is isomorphic to $M(\ax)$ for some action of $\ax:\Gamma\curvearrowright (\Omega,\mu)$ on a nonstandard probability space. This nonstandard point realization $\ax$ is one definition of the ultraproduct of actions.
\begin{dfn}
    We say that a (possibly nonstandard) action $\ax$ is an \textbf{ultraproduct} of actions $\ip{\ax_i: i\in\N}$ with respect to an ultrafilter $\s U$ if $M(\ax)\cong \lim_{\s U} M(\ax_i)$. And we write $\ax=\lim_{\s U} \ax_i$. For $\ax$ a (standard) pmp action we call $\ax^{\s U}:=\lim_{\s U} \ax$ the \textbf{ultrapower} of $\ax$.
\end{dfn}

Continuous model theory has isolated a useful class of properties preserved by ultraproducts. The rest of this subsection is stated only for the language of actions, but holds much more generally \cite{Survey}.

\begin{dfn} \label{dfn: language}
    A \textbf{term} is a function\footnote{Technically, a term isn't a function but a formal symbol with a canonical interpretation as a function in any $*$-algebra. You can find detail in the survey.} from (some power of) a $*$-algebra to itself built from the $*$-algebra operations and the application of automorphisms\footnote{You can think of terms as nouns in the language of actions.}. In more detail:
    \begin{itemize}
        \item For each $i$, the function $\tau(x_1,...,x_n)=x_i$ is a term
        \item The constant function $\tau=\bf 1$ is a term.
        \item If $\tau_1$ and $\tau_2$ are terms, $\gamma\in \Gamma$, and $\lambda\in \C$, the following are all terms:
        \[\tau_1+\tau_2,\;  \tau_1\tau_2,\;  \tau_1^*, \; \gamma\cdot \tau_1, \; \lambda\tau_1.\]
    \end{itemize}

    A \textbf{quantifier free formula} is a function from a $*$-algebra on a Hilbert space to $\R$ built by taking traces of terms and using continuous connectives\footnote{You can think of formulas as predicates in the language of actions; the value of a formula is how far it is from being true.}:
    \begin{itemize} 
    \item If $\tau_1(x_1,...,x_n)$ and $\tau_2(x_1,...,x_n)$ are terms, then $\Re\bigl(\ip{\tau_1(x_1,...,x_n), \tau_2(x_1,...,x_n)}\bigr)$ is a quantifier free formula, as is the imaginary part of the inner product.
    \item If $\phi_1,\phi_2,...,\phi_n$ are quantifier free formulas and $c: \R^n\rightarrow\R$ is uniformly continuous, then $c(\phi_1,...,\phi_n)$ is also a quantifier free formula.
    \end{itemize}
\end{dfn}

For example, $fg^*$ and $f-g$ are two-variable terms, $d(f,g)=\sqrt{\ip{(f-g),(f-g)}}$ is a two-variable quantifier free formula. We'll sometimes speak of complex-value functions as being formulas, in which case we mean the real and imaginary parts. In continuous model theory, the quantifiers are $\sup$ and $\inf$. One annoying technicality is that a $\sup$ or $\inf$ might be infinite. Model theorists deal with this by truncating all formulas and working in the unit ball of $L^2(X,\mu)$. We'll add a condition on what sort of formulas we can quantify over. This all just amounts to some inconvenient bookkeeping.

\begin{dfn}
    An \textbf{existential formula}, or \textbf{$\Sigma_1$-formula}, is a function of the form
    \[\inf_{x_1,...,x_n\in M(\ax)} \phi(x_1,...,x_n, y_1,...,y_n)\] where $\phi$ is a quantifier free formula which is bounded below in $\bar x$. (If we think of the value of $\phi$ as being its distance from being true, then $\inf_x \phi(x)$ is small when there exists some value that makes it small.) We write $\Sigma_1$ for the set of all existential formulas.

    A \textbf{first order formula} is function of the form:
    \[\inf_{x_1}\sup_{x_2}\inf_{x_3}\sup_{x_4}... \sup_{x_{2n}}\phi(\bar x, \bar y)\] where $\phi$ is quantifier free, and each subformula is appropriately bounded.

    A \textbf{sentence} is a formula with no free variables.
\end{dfn}

So, for instance, $\inf_f \tr\bigl((f^2-g)^*(f^2-g))$ is a one place $\Sigma_1$ formula. Note that you can canonically interpret a first order formula $\phi(x_1,...,x_n)$ in any structure $M(\ax)$ which contains $x_1,...,x_n$. In particular, a sentence gives some numerical invariant of an action.

\begin{dfn}
    For a first order formula
    \[\psi(\bar y)=\inf_{x_1}\sup_{x_2}\inf_{x_3}\sup_{x_4}... \sup_{x_{2n}}\phi(\bar x, \bar y)\] an action $\ax$, and $f_1,...,f_m\in M(\ax)$, we write
    \[\psi^{\ax}(\bar f)=\inf_{x_1\in M(\ax)}\sup_{x_2\in M(\ax)}\inf_{x_3\in M(\ax)}\sup_{x_4\in M(\ax)}... \sup_{x_{2n}\in M(\ax)}\phi(\bar x, \bar f).\]

    Sometimes it is convenient to write $M(\ax)\vDash \phi(x_1,...,x_n)\leq r$ to mean $\phi^{\ax}\leq r$.
\end{dfn}

For instance, (bear in mind $\Gamma=\ip{\gamma_1,...,\gamma_n}$) \[M(\ax)\vDash\inf_x \max(1-|x|, |x-\gamma_1\cdot x|,...,|x-\gamma_n\cdot x|, \ip{x,\bf{1}})=0\] exactly when $\ax$ admits non-trivial approximately invariant vectors.

The following proposition is useful, and the proof is a simple induction on the length of $\phi$.

\begin{lem}\label{lem:qfpreserve}
    If $f: M(\ax)\rightarrow M(\bx)$ is an equivariant isometric embedding of $*$-algebras, and $\phi$ is quantifier free, then $\phi^{\ax}(x_1,...,x_n)=\phi^b\bigl(f(x_1),...,f(x_n)\bigr)$
\end{lem}

Sentences in $M(\ax)$ detect a great deal about the action. For instance, since any pmp graph with degree at most $2$ is measurable $3$-colorable:
\begin{align*} \ax\mbox{ is free }\lra & (\forall \gamma\in \Gamma,\epsilon>0) (\exists A_1,A_2,A_3\subseteq X)\; A\cup B\cup C=X \;\& \; A_i\cap \gamma\cdot A_i=0 \\
\lra & (\forall \gamma)\; M(\ax)\vDash \inf_{f_1,f_2,f_3} \max\bigl(\bigl|(f_1+f_2+f_3)-\bf 1\bigr|, |f_i^2-f_i|, |(\gamma\cdot f_i)f_i|  \bigr)=0\end{align*} (It happens to be the case that the infimum above is always realized. That isn't usually the case.)

So, knowing the value of every sentence will tell you whether the action is free or not. On the other hand, ergodicity is not detected by the first order theory. Indeed, any two free pmp actions of $\Z$ have the same first order theory, but there are ergodic and nonergodic free $\Z$-actions. The problem is that the existence of almost invariant sets means there is no first-order way to express $\gamma\cdot f\not= f$ in the definition:

\[\ax\mbox{ is ergodic }\lra (\forall f)\; \bigl(f=\tr(f)\bf 1\mbox{ or }(\exists \gamma\in \Gamma) \; \gamma\cdot f\not= f\bigr).\]

In fact, one can view strong ergodicity as the weakest strengthening of ergodicity which is first-order expressible. We have a name for the ensemble of numerical invariants given by first order sentences:

\begin{dfn}
    Let $\s L$ be the set of first order sentences.  The \textbf{first order theory} of an action $\ax$ is 
    \[Th(\ax)=\ip{\phi^{\ax}: \phi\in \s L}\] The \textbf{$\Sigma_1$-theory} of $\ax$ is  \[Th_{\Sigma_1}(\ax)=\ip{\phi^{\ax}: \phi\in \Sigma_1}.\]
\end{dfn}

\textit{A priori}, the elementary theory of an action lives in an inseparable product space: the size of $\s L$ is $|\R|$ and a theory lives in $\R^{\s L}$. But, there is a countable subset of $\s L$ that suffices for all pratical purposes:
\begin{prop}[{\cite[Remark 3.11]{Survey}}] \label{prop:densesentences}
    There is a countable subset $\tilde {\s L}\subseteq \s L$ so that, for any $\epsilon>0$ and $\phi\in \s L$ there is $\tilde \phi\in \tilde {\s L}$ so that, for all actions $\ax$,
    \[|\phi^{\ax}-\tilde \phi^{\ax}|<\epsilon.\]
\end{prop} 

The $\Sigma_1$- and first order theories naturally give rise to equivalence relations and topologies on the corresponding quotients.

\begin{dfn}
   We say two actions, $\ax$ and $\bx$ are \textbf{elementary equivalent} if, for all first order sentences $\phi$, $\phi^{\ax}=\phi^\bx$. In this case we write
   \[\ax\equiv \bx\] Similarly, if $\ax$ and $\bx$ agree on all $\Sigma_1$-sentences, we say they are $\Sigma_1$-equivalent and write
   \[\ax \equiv_{\Sigma_1}\bx\]
   
   We say an embedding $\iota:M(\ax)\hookrightarrow M(\bx)$ is an \textbf{elementary embedding} if, for any $x_1,...,x_n\in M(\ax)$ and first order formula $\phi$
   \[\phi^{\ax}(x_1,...,x_n)=\phi^b\bigl(\iota(x_1),...,\iota(x_n)\bigr).\] If such an embedding exists, we write
   \[M(\ax)\preccurlyeq M(\bx)\]
\end{dfn}
The notation $M(\ax)\preccurlyeq M(\bx)$ is somewhat misleading (though, standard). This is a much stronger condition than $M(\ax)\equiv M(\bx)$. 

For example, \L o\'s's theorem (Theorem \ref{thm:los} below) says that the diagonal embedding, \[x\mapsto [\ip{ x:i\in\N}]_{\s U},\] of $M(\ax)$ into its ultrapower $M(\ax)^{\s U}$ is elementary.

\begin{dfn}
   Two actions, $\ax$ and $\bx$, are $\Sigma_1$-\textbf{equivalent} if they agree on all $\Sigma_1$-sentences. In this case, we write
   \[M(\ax)\equiv_{\Sigma_1} M(\bx).\] 

   A sequence of actions $\ip{\ax_i:i\in\N}$ is \textbf{elementary Cauchy} (or $\Sigma_1$ \textbf{Cauchy}) if for every first order (or existential) sentence $\phi$, $\phi^{\ax_i}$ converges. The sequence ($\Sigma_1$- or) \textbf{elementary converges} if there is some (standard) pmp action $\ax$ so that $\phi^{\ax_i}$ converges to $\phi^\ax$ for all (existential) $\phi$.
\end{dfn}

That is, a sequence of action elementary converge if their theories converge in the product topology. We'll see later that the space of elementary equivalence classes with the topology of elementary convergence is compact, and by Proposition \ref{prop:densesentences} it's separable.

The following two theorems are fundamental and can be found in the survey \cite{Survey}.
    
\begin{thm}[Continuous \L o\'s's theorem, {\cite[Thm 5.4]{Survey}}] \label{thm:los}
For any first order formula $\phi$,
\[\phi^{\lim_{\s U}\ax_i}=\lim_{\s U} \phi^{\ax_i}.\] In fact, for any $[x^1],...,[x^n]\in \lim_{\s U} M(\ax_i)$,
\[\bigl(\lim_{\s  U} M(\ax_i)\bigr)\vDash \phi\bigl([x^1],...,[x^n]\bigr)\leq\epsilon\; \lra \; (\exists A\in \s U) (\forall i\in A)\; M(\ax_i)\vDash \phi(x^1_i,...,x^n_i)\leq \epsilon.\] 
\end{thm}
It follows that ultraproducts are saturated in the following sense:
\begin{cor}[Saturation of ultraproducts] Let $T=\{\phi_i: i\in\N\}$ be a countable set of formulas so that any finite subset can be satisfied in $\lim_{\s U} M(\ax_i)$, i.e.
    \[(\forall i,\epsilon)(\exists c_1,...,c_n\in \lim_{\s U} M(\ax_i))\; \max_{j\leq i}\phi_j(c_1,...,c_n)<\epsilon\]
    then $T$ can be satistfied all at once, i.e.
    \[(\exists c_1,...,c_n\in \lim_{\s U} M(\ax_i))(\forall i)\; \phi_i(c_1,...,c_n)=0.\]
    In particular, if $\lim_{\s U}M(\ax_i)\vDash \inf_x \phi(x)=0$ then this infimum is realized.
\end{cor}
\begin{proof} This is a straightforward diagonalization: we can find a nested sequence of large sets where the $n^{th}$ set has a witness to the first $n$ formulas being small. We can then take the ultralimit of a sequence drawn from these witnesses. We'll handle the one variable case in detail. The general case is just a cumbersome notational change. 

Suppose that for each $i$, there is a sequence $c^i$ so that \[\lim_{\s U} M(\ax_k)\vDash \max_{j<i}\phi_j\bigl([c^i]\bigr)<1/i.\] For each $i$ we can find $A_i\in \s U$ with $A_i\subseteq [i,\infty)$ such that for all $j<i$ and $k\in A_i$, $\phi^{a_k}_j(c^i_k)<\frac{1}{i}.$ Let $B_i=\cap_{k=1}^i A_k$. Then 
    \[\bigcap_{i\in\N} B_i=\emptyset\mbox { and }\quad (\forall i\in \N)\; B_i\in \s U,\quad  B_i\supseteq B_{i+1}. \] So, we can define $c_k$ by $c_k=c^i_k$ if $k\in B_i\sm B_{i+1}$. Since the intersection of the $B$s vanishes, this defines $c_k$ for all $k\in B_1$, and we can define $c_k$ however we like on all other values of $k$.
    For any $j$, for all $k>j$ with $k\in B_k$, $\phi_j^{\ax_k}(c_k)=\phi_j^{\ax_k}(c^i_k)$ for some $i>k$ with $k\in A_i$. Thus, $\phi_j^{a_k}(c_k)<\frac{1}{k}.$ So by \L o\'s,'s theorem, for any $j$, $i\in\N$, 
    \[\phi_j^{\lim_{\s U} M(\ax_i)}([c])=\lim_{\s U} \phi_j(c_i)<1/i.\]
    
\end{proof}
The L\"owenheim--Skolem theorem for continuous model theory lets us extract standard pmp actions from ultraproducts.
\begin{thm}[Continuous L\"owenheim--Skolem theorem, {\cite[Thm 7.3]{Survey}}] \label{thm:lowenheimskolem}
    For any countable subset $A\subseteq (\lim_{\s U} M(\ax_i))$, there is a separable model $M(\ax)$ and an elementary embedding $\iota: M(\ax)\hookrightarrow (\lim_{\s U} M(\ax_i))$ with $A\subseteq \im(\iota)$.
\end{thm}

These theorems imply that elementary equivalence classes of separable structures are compact in the topology of elementary convergence.

\begin{cor}\label{thm:logiccompactness}
    Suppose that $\ip{\ax_i: i\in\N}$ is a sequence of actions. Then there is a subsequence $\ip{\ax_{f(i)}: i\in \N}$ and an action $\ax$ on a standard space so that $Th(\ax_{f(i)})$ converges to $Th(\ax)$. (And similarly with elementary convergence and theory replaced by $\Sigma_1$ convergence and theory).
\end{cor} 
It follows that any Cauchy sequence is convergent.

\begin{proof}
    By the L\"owenheim--Skolem theorem, we can find a standard action $\ax$ so that $M(\ax)\preccurlyeq \lim_{\s U} M(\ax_i)$. Then, for a sentence $\phi$, 
    \[\phi^{\ax}=\phi^{\lim_{\s U} \ax_i}=\lim_{\s U} \phi^{\ax_i}.\] Now we'll extract the desired subsequence by diagonalizing against a countable dense set of sentences. Let $\tilde {\s L}$ be the countable subset of first order sentences from Prop \ref{prop:densesentences}. Say $\phi_1,...,\phi_n$ are the first $n$ sentences in $\tilde{\s L}$. From the above calculation, the set of indices $i$ so that
    \[|\phi_1^{\ax}-\phi_1^{\ax_i}|,...,|\phi_n^{\ax}=\phi_n^{\ax_i}|<\frac{1}{n}\] is in $\s U$. So, certainly this set is nonempty, and we can let $f(n)$ be the least such index bigger than $n$. Then, for any sentence $\phi\in \tilde {\s L}$, $\phi^{\ax_{f(i)}}$ converges to $\phi^{\ax}$. And by the density of $\tilde {\s L}$ in $\s L$, this is true for all first order sentences $\phi$. 

    Lastly, note that elementary convergence implies $\Sigma_1$-convergence.
\end{proof}

We can also use these tools to replace any action by an elementary equivalent action where a given $\sup$ or $\inf$ is realized.

\begin{prop}
    For any pmp action $\ax$ and first order formula $\phi$, if \[\ax\vDash \inf_x \phi(x)=0,\] then there is an elementary equivalent action $\bx$ with some $f\in M(\bx)$ so that $\phi^{\bx}(f)=0$.
\end{prop}
\begin{proof}
    By saturation, there is some $f\in M(\ax)^{\s U}$ so that $M(\ax)^{\s U}\vDash \phi(f)=0$. By the L\"owenheim--Skolem theorem, we can find some pmp action $\bx$ and an elementary embedding $g: M(\bx)\rightarrow M(\ax)^{\s U}$ with $f\in \im(g)$. By elementarity, $\phi^{M(\bx)}(f)=\phi^{M(\ax)^{\s U}}(f)=0.$
\end{proof}

We will also need a fact from model theory with obvious connections to weak containment. The proof uses a continuous version of the diagram argument from classical model theory.

\begin{thm}\label{thm:embeddingthm}
    A structure $M(\ax)$ embeds into an ultrapower $M(\bx)^{\s U}$ if and only if for every $\Sigma_1$-sentence $\phi$, $\phi^b\leq \phi^{\ax}$.
\end{thm}
\begin{proof}[Proof sketch.]
For the left to right direction, suppose that $f$ is an embedding of $M(\ax)$ into $M(\bx)^{\s U}$. For any quantifier free $\phi$, by Lemma \ref{lem:qfpreserve}, if $x_0$ witnesses that $(\inf_x \phi(x))^{M(\ax)}$ is small, then $f(x_0)$ witnesses that $(\inf_x\phi(x))^{M(\bx)^{\s U}}$ is small. And, $M(\bx)$ and $M(\bx)^{\s U}$ are elementary equivalent, so 
\[\bigl(\inf_x \phi(x)\bigr)^{M(\ax)}\geq \bigl(\inf_x \phi(x)\bigr)^{M(\bx)^{\s U}}=\bigl(\inf_x \phi(x)\bigr)^{M(\bx)}.\]

For the other direction, let $\s A=\{\ax_i:i\in\N\}\subset M(\ax)$ be a countable dense substructure. It suffices to to give an isometric embedding of $\s A$ into $M(b)$. Let $\s L_{\s A}$ be the language of tracial {von Neumann} algebras equipped with an action expanded by a countable set of constants $\{c_i:i\in\N\}$. Define the atomic diagram of $\s A$ to be the following set of quantifier free sentences in $\s L_{\s A}$:
\[\Delta(\s A)=\{\phi(c_1,...,c_n): M(\ax)\vDash \phi(a_1,...,a_n)=0\}.\]
Then, we need to show that there is some assignment $c_i\mapsto a'_i\in M(\bx)^{\s U}$ which gives the correct value to $\Delta(\s A)$ (meaning $\phi^{\bx^{\s U}}(a'_1,...,a'_n)=0$ for each $\phi\in \Delta(\s A)$). Since the ultraproduct is saturated and $\Delta(\s A)$ is closed under taking maxima and minima, it suffices to do this for every sentence in $\Delta(\s A)$ individually. By our assumption, \[M(\bx)^{\s U} \vDash \inf_{x_1,...,x_n}\phi(x_1,...,x_n)=0\] for every $\phi\in \Delta(\s A)$. And, again by saturation, this infimum is realized in $M(\bx)^{\s U}$. 
\end{proof}
\subsection{Definitions from combinatorics}

Local-global convergence was first introduced in the context of sparse graph limits. Here we give a restricted definition which applies to free actions of locally finite groups. We give the more general definition later.

\begin{dfn}
    For a marked group $\Gamma$ with identity element $e$, and $r,k\in \N$, the corresponding local labelling space \emph{for free actions} is
    \[\bb L_f(r,k,\Gamma)=[k]^{B_r(e)}.\] (In this notation, $\Gamma$ indicates that we are working with labellings of $\Gamma$ actions, and $f$ means they are free. We will relax both of these conditions later.)
    
    Note that $\bb L_f(r,k,\Gamma)$ a finite set, so the space of probability measures $P_1\bigl(\bb L_f(r,k,\Gamma)\bigr)$ is compact under any reasonable metric and any two reasonable metrics will be equivalent. For concreteness, we equip this space with the total variation distance.
    \[d_{TV}(\mu,\nu)=\sum_{f\in \bb L_f(r,k,\Gamma)} \bigl|\mu(\{f\})-\nu(\{f\})\bigr|.\]
    For a free pmp action $\ax:\Gamma\curvearrowright (X,\mu)$, a measurable labelling $f:X\rightarrow [k]$, and $r\in \N$, the \textbf{local statistics} of $f$ to radius $r$ is the measure $\mu^r(f)$ on $\bb L_f(r,k,\Gamma)$ given by sampling a random radius $r$ neighborhood in $f$:
    \[\mu^r(f)(g):=\mu\bigl(\{x: (\forall \gamma\in B_r(e))\; f(\gamma\cdot x)=g(\gamma)\}\bigr).\]

    The set of $(r,k)$-local statistics of the action $\ax$ is the closure of the set of local statistics of $k$-labellings of $X$. That is,
    \[{\s L'_{r,k}}(\ax):={\{\mu^k(f): f\in [k]^X\mbox{ measurable}\}}\]
    \[\s L_{r,k}(\ax)=\overline{\s L'_{r,k}(\ax)}\]
\end{dfn}

Note that if $X$ is finite, then $\s L'_{rk}(\ax)$ is finite. So, this closure doesn't add any new points. Similar observations will hold for finite graphs in later sections.

The $(r,k)$-local statistics of an action detect many dynamical and combinatorial parameters of an action. For instance: 
\begin{dfn} If $\Gamma=\ip{\gamma_1,...,\gamma_n}$ is a marked group, and $\ax:\Gamma\curvearrowright (X,\mu)$ is a pmp action, the \textbf{independence ratio} of $\ax$ is the suprememum of the measures of independent sets in the Schreier graph of $\ax$. That is, for free actions,
\[\alpha_\mu(\ax):= \sup\bigl\{\mu(A): A\subseteq X,\mbox{ and for }i=1,...,n, (\gamma_i\cdot A)\cap A=\emptyset\bigr\}\]
(for non-free actions, we need to consider $(A\cap \gamma_i\cdot A)\setminus \fix(\gamma_i)$ to ignore self-loops).  
\end{dfn} This is detected by $\s L_{1,2}(\ax)$:
\begin{align*}
    \alpha_\mu(\ax)=&\sup\{\mu(A): A\subseteq X, (\gamma_i\cdot A)\cap A=\emptyset\}\\
                =&\sup\bigl\{\nu(\{f: f(e)=1\}): \nu\in \s L_{r,k}(\ax), \nu\bigl(\{f: (\exists i)\; f(e)=f(\gamma_i)=1\}\bigr)=0\bigr\}
\end{align*}
Similarly, we can compute the matching ratio of an action of $\Gamma$ as above from $\s L_{1n}(\ax)$ (where $n$ is the number of generators) by identifying a matching with the function that takes a vertex to the edge it's matched along.

More generally, the maximum density of a set on which a measurable labelling of the Schreier graph will satisfy a locally checkable constraint (in the sense of \cite{gridsSurvey}) is detected by $\s L_{r,k}$, where $k$ is the label set and $r$ is the radius needed to check the local constraint. 

An action $\ax$ is said to local-global contain another action $\bx$ if the local statistics of any labelling of $\bx$ can be well approximated by the statistics of an $\ax$ action. 
\begin{dfn}
    For two free pmp actions $\ax,\bx$ of $\Gamma$, $\ax$ \textbf{local-global contains} $\bx$ if for each $r,k$ the set of $(r,k)$-local statistics of $\ax$ contain those of $\bx$. In symbols:
    \[b\preccurlyeq_{lg} \ax:\lra (\forall r,k)\; \s L_{rk}(\bx)\subseteq \s L_{rk}(\ax).\] (Recall that the definition of $\s L_{rk}(\ax)$ involves taking a closure).
\end{dfn}
So, for instance, any independent set in $\bx$ of measure $\alpha$ gives rise to independent sets in $\ax$ with measure close to $\alpha$.

\begin{dfn}
    We say two actions are \textbf{local-global equivalent} if they have the same local statistics. That is
    \[\ax\equiv_{lg} \bx:\lra a\preccurlyeq_{lg} b\mbox{ and } b\preccurlyeq_{lg} a.\]
    A sequence of actions $\ip{\ax_i:i\in\N}$ is \textbf{local-global Cauchy} if, for every $r,k\in \N$, the sequence $\ip{\s L_{rk}(\ax_i):i\in\N}$ converges in the Hausdorff metric. And, the sequence \textbf{local-global converges} to $\ax$ if $\s L_{rk}(\ax_i)$ converges to $\s L_{rk}(\ax)$ for each $r,k$.
\end{dfn}
Note that since $P_1\bigl(\bb L_f(r,k,\Gamma)\bigr)$ is compact, so is its Hausdorff space. Thus the space of local-global equivalence classes of actions under local-global convergence is compact provided every Cauchy sequence converges. This is indeed the case, as we shall see in Theorem \ref{thm:axnconvergence}.
\subsection{Equivalences for free actions}

The notions from the preceding subsections turn out to all coincide for free actions. The main point is that the local statistics of a labelling are expressed by an existential formula. To give a clean proof of this, we'll use a handy fact about continuous first order definability.

\begin{lem}[{\cite[Theorem 9.17]{Survey}}]\label{lem:definable}
Say $D\subseteq M(\ax)^n$ is definable if there is a first order formula $\phi(\bar x)$ and increasing, uniformly continuous functions $\alpha, \beta:\R\rightarrow \R$ so that $\alpha(0)=\beta(0)=0$ and for any $\bar f\in M(\ax)^n$
\begin{enumerate}
    \item $\alpha(\phi(\bar f))\leq d(\bar f, D)\leq \beta(\phi(\bar f))$
    \item $\phi(\bar f)=0$ iff $x\in D$.
\end{enumerate} If $\phi$ is definable, then for any first order formula $\psi(\bar x, \bar y)$, \[\inf_{\bar x\in D} \psi(\bar x,\bar y)\] is also first order. (likewise for $\sup_{x\in D}\phi(x).$)

Further, if $D$ is definable by a $\Sigma_1$-formula and $\phi$ is $\Sigma_1$, then $\inf_{x\in D} \phi(x)$ is $\Sigma_1.$
\end{lem}

Roughly this is because we can find a re-weighting factor $\delta$ so that 
\[\inf_{\bar x\in D} \psi(\bar x,\bar y)=\inf_{\bar x\in M(\ax)^n} \psi(\bar x,\bar y)+\delta(d(\bar x, D)),\] and use points $(1,2)$ above to swap out $\phi$ for $d(\bar x, D)$.

\begin{prop}
    For any $\mu\in P_1\bigl(\bb L_f(r,k,\Gamma)\bigr)$, 
    \[\phi_{\mu}(f):= \inf_{g:X\rightarrow [k]} |f-g|+d_{TV}(\mu^{r}(g),\mu)\] is given by a $\Sigma_1$ formula.
\end{prop}
\begin{proof}
    The main idea is that the set of characteristic functions is definable in the sense of the previous proposition: the distance from $f$ to the nearest characteristic function is uniformly bounded above and below by $\max(d(f, f^2), d(f,f^*)).$ So, we can quantify over $[k]$-value functions by instead quantifying over sums of $k$ characteristic function.

    For $g: X\rightarrow [k]$, let $e_i=\bf 1_{g\inv[i]}$, so $g=\sum_{i=1}^k ie_i$. Then,
    \begin{align*} d_{TV}(\mu^r(g),\mu)=& \sum_{h\in k^{B_r(e)}} |\mu(h)-\mu^{r}(g)(h)| \\                                              =& \sum_{h\in k^{B_r(e)}} \bigl|\mu(h)-\tr\bigl(\prod_{\gamma\in B_r(e)}\gamma\inv\cdot e_{h(\gamma)}\bigr) \bigr| \end{align*}
Abbreviate this last sum by $\psi(e_1,...,e_k)$. Let $D\subseteq M(\ax)^k$ be the set of $k$-tuples of disjoint characteristic functions: 
\[D=\{(e_1,...,e_k): (\forall i,j)\, e_i^*=e_i, e_ie_j=\delta_{ij} e_i\}.\]
Since this is definable by an existential formula, the above lemma tells us,
\begin{align*} \phi_{\mu}(f)=&\inf_{e_1,...,e_k\in P} d(f, \sum_{i=1}^k ie_i)+\psi(e_1,...,e_k)\in \Sigma_1 .\end{align*}
   
\end{proof}
\begin{cor}
    For all $r,k$, $\s L_{rk}(\lim_{\s U}\ax_i)=\lim_{\s U} \s L_{rk}(\ax_i)$, where the latter limit is taken in the Hausdorff metric.
\end{cor}
\begin{proof}
    The following characterization of ultralimits in the Hausdorff metric is a straightforward calculation: for $\ip{E_i:i\in\N}$ a sequence in $K(X)$ and $x\in X$,
    \[x\in \lim_{\s U} E_i \lra (\exists x_i \in E_i)\;\lim_{\s U} d(x,x_i)=0.\]
    Using $\phi_\mu$ from above,
     \[\lim_{\s U}d_{TV}\bigl(\mu,\s L_{r,k}(\ax_i)\bigr)=\lim_{\s U} \inf_{f} \psi^{\ax_i}_\mu(f)=\inf_f \psi^{\lim_{\s U} \ax_i}_\mu(f)=d_{TV} \bigl( \mu, \s L_{rk}(\lim_{\s U} \ax_i)\bigr)\]
    So, $\mu\in \lim_{\s U}\s L_{rk}(\ax_i)$ if and only if the left hand side above is $0$ if and only if the right hand side is $0$ if and only if $\mu\in \s L_{rk}(\lim_{\s U}\ax_i)$.
\end{proof}

We can now prove the equivalences:

\begin{thm} \label{thm:freeequivalents}
    For free pmp actions $\ax,\bx$ of a group $\Gamma$, the following are equivalent:
    \begin{enumerate}
        \item $\bx\preccurlyeq_w \ax$
        \item $\bx\preccurlyeq_{lg} \ax$
        \item $(\forall \phi\in \Sigma_1)\; \phi^{\ax}\leq \phi^\bx$.
    \end{enumerate}
\end{thm}
\begin{proof}
    The equivalence of $(3)$ and $(1)$ is a restatement of Theorem \ref{thm:embeddingthm}. The implication from $(3)$ to $(2)$ is essentially the above proposition. In more detail, suppose $\mu$ is some measure in $\s L_{r,k}(\ax)$ and let $\phi_\mu$ be as above. Then, $\mu\in \s L_{r,k}(\ax)$ means $\bigl(\inf_f \phi_\mu(f)\bigr)^{\ax}=0$, so $\bigl(\inf_f \phi_{\mu}(f)\bigr)^{b}=0$. Thus $\mu\in \s L_{r,k}(\bx).$
    
    To get $(2)$ implies $(1)$, we need to show that every $\Sigma_1$ fact about $\bx$ is implied by its local statistics. To do this, we show that these local statistics are enough to find a copy of a dense countable substructure as in the proof of Lemma \ref{thm:embeddingthm}. Let $\s B_0$ be some countable basis for $X$, and let $\s B=\{b_i:i\in\N\}$ be the (countable, not closed) $\Q[i]$-subalgebra of $L^2(Y,\nu)$ spanned by the characteristic functions of sets in $\s B_0$.

    We want to find $\{a_i:i\in\N\}\subseteq M(\ax)^{\s U}$ so that $(b_i\mapsto a_i)$ is an isomorphism. By saturation, it suffices to do this up to $\epsilon$ for every finite set of $a_i$'s. This boils down to checking that for any list of characteristic functions $e_1,...,e_n\in M(\bx)$ and group elements $\gamma_1,...,\gamma_n$, there are $f_1,...,f_n\in M(\ax)$ with 
    \[\left|\tr\biggl(\prod_{i=1}^n\prod_{j=1}^n \gamma_j\cdot f_i \biggr) -\tr\biggl(\prod_{i=1}^n\prod_{j=1}^n  \gamma_j\cdot e_i \biggr)\right|<\epsilon. \] These traces are all detected by the local statistics of the labelling $f(x)=\bigl(e_1(x),...,e_n(x)\bigr)$. So local-global containment says we can satisfy all these equations.
\end{proof}
As a corollary, we can boost $\Sigma_1$-equivalence of free actions to elementary equivalence.

\begin{cor}\label{cor:sigma1implieselementary}
    For any free pmp actions $\ax,\bx$, if $\ax\equiv_{\Sigma_1}\bx$, then $\ax\equiv\bx.$
\end{cor}
\begin{proof}We may assume $\ax$ and $\bx$ are actions on the same space since the $\Sigma_1$ theory determines the number and weight of the atoms in the underlying measure space.

    Suppose $\ax\equiv_{\Sigma_1}\bx$. Then, by the theorem, $\ax$ is weakly equivalent to $\bx$. It follows that $\ax$ can be uniformly approximated by conjugates of $\bx$, and vice-versa \cite{BurtonKechrisSurvey}. Induction on the number of quantifiers shows that first order sentences are continuous with respect to the uniform topology on actions.
\end{proof} This corollary is also true for non-free actions, but we need to show that $\Sigma_1$-equivalence still implies weak equivalence for free actions. We will do that in the next subsection.

 We also get equivalence of the relevant topologies in the case of free actions:
\begin{thm}\label{thm:axnconvergence}
    For a sequence of free pmp actions $\ax_i:\Gamma\curvearrowright X_i$, the following are equivalent
    \begin{enumerate}
        \item $\ip{\ax_i:i\in\N}$ $\Sigma_1$-converges to $\ax$
        \item $\ip{\ax_i:i\in\N}$ elementary converges to $\ax$
        \item $\ip{\ax_i:i\in\N}$ local-global converges to $\ax$.
    \end{enumerate}
\end{thm}
\begin{proof} Clearly elementary convergence implies $\Sigma_1$ convergence. 

    If $\ax_i$ $\Sigma_1$-converges to $\ax$, then for every $r,k\in\N$ and $\mu\in P_1\bigl(\bb L_f(r,k,\Gamma)\bigr)$, $\bigl(\inf_{f}\phi_\mu(f)\bigr)^{\ax_i}$ converges to $\bigl(\inf_f \phi_\mu(f)\bigr)^{\ax}$. So, in particular $d_{TV}\bigl(\mu, \s L_{rk}(\ax_i)\bigr)$ converges to $d_{TV}\bigl(\mu, \s L_{rk}(\ax)\bigr)$. Thus $\s L_{rk}(\ax_i)$ converges to $\s L_{rk}(\ax)$, and $\ax_i$ local-global converges to $\ax$.

    Lastly, suppose $\ax_i$ local-global converges to $\ax$. Since the space of elementary theories is compact, it suffices to show that any elementary convergent subsequence of $\ip{\ax_i:i\in\N}$ converges to $\ax$. By the above, the limit of any subsequence will be local-global equivalent to $\ax$, so will be $\Sigma_1$-equivalent to $\ax$, as desired.
\end{proof}
\begin{cor}
    The topology of local-global convergence for actions is compact.
\end{cor}
\begin{proof}
    Local-global convergence agrees with elementary convergence, which is compact by Corollary \ref{thm:logiccompactness}.
\end{proof}

\subsection{Subtleties for non-free actions}
There are essentially two ways to generalize the forgoing definitions beyond free actions. According to one approach-- more typical in dynamics --the definitions essentially remain unchanged (you will note none of them really require free-ness). And the equivalences are also unchanged. For instance, if $\bx_0$ is the trivial action of $\Gamma$ on a one-point space, then every action contains $\bx_0$, $M(\bx_0)$ is the algebra of constant functions, and the local statistics of $\bx_0$ are similarly boring but still measure labellings of $B_r(e)$.

The other approach more closely aligns with the original definition of local-global convergence for graphs. The local statistics of $f$ record not only the pull-back of  $f$ to $B_r(e)$ at each point $x$, but also the isomorphism type of $B_r(x)$ in the Schreier graph. Then, almost no other action will weakly contain the trivial action since no nontrivial action has the trivial graph showing up as a ball of radius $1$ with probability $1$. In the language of weak equivalence, this corresponds to asking that factor maps preserve stabilizers. In terms of model theory, this corresponds to adding constants to our language for the characteristic functions
\[\bf 1_{\fix(\gamma)}(x)= \left\{\begin{array}{cc}
    1 & \gamma\cdot x=x  \\
    0 & else
\end{array}\right..\]

We will leave the definition of weak containment untouched (as this is somewhat more standard in ergodic theory), and we will modify the definition of local-global convergence to agree with local-global convergence of the Schreier graph of the action. We'll have it both ways for model theory-- we will speak of existential theories in two different languages.

The subtleties outlined above are only an issue for containment, though. The characteristic function $\bf 1_{\fix(\gamma)}$ is always definable in the sense of Lemma \ref{lem:definable}, and any $\Sigma_1$ sentence in this expanded language is first order in the original language. As with free actions, $\Sigma_1$ equivalence implies elementary equivalence, which then implies equivalence in the expanded language. So, all our notions of equivalence and convergence align.

We want the local statistics of a labelling of a $\Gamma$-Schreier graph to be supported on finite rooted labelled graphs equipped with a labelling of the edges by generators of $\Gamma$. In the case free actions, the finite graph was always $B_r(e)$ and the edge-labelling was implicit-- $x(\gamma\cdot x,x)=\gamma$. It will be convenient to code the edge-labelled graphs appearing as neighborhoods in $\Gamma$-Schreier graphs as quotients of $B_r(e)$ and again leave the edge-labelling implicit.

\begin{dfn}
   A \textbf{quotient $r$-ball} for $\Gamma$ is an equivalence relation on $B_r(e)$. Given an action $\ax:\Gamma\curvearrowright X$ on a set $X$ and $x\in X$, the quotient $r$-ball of $x$ is the equivalence relation $B_r(x,\ax)$, with classes denoted $[\gamma]_x$, given by
   \[[\gamma]_x=[\delta]_x:\lra \gamma\cdot x=\delta\cdot x.\]
   We write $\bb G(r,\Gamma)$ for the set of quotient $r$-balls in $\Gamma$. (Not all possible quotient $r$-balls arise as $B(x,a)$ for some action $\ax$, but no matter).

   For $r,k\in \N$, the corresponding label space \emph{for general actions} is 
   \[\bb L(r,k,\Gamma)=\{(E, f): E\in \bb G(r, \Gamma), f: B_r(e)/E\rightarrow [k]\}. \] (Again, $\Gamma$ in this notation tells us we're studying $\Gamma$ actions. This will be relaxed to a simple degree bound later).

   Given an action $\ax:\Gamma\curvearrowright(X,\mu)$ and $f:X\rightarrow [k]$, the $r$-local statistics of $f$ is the measure $\mu^r\in P_1\bigl(\bb L(r,k,\Gamma)\bigr)$ on labellings of $r$-balls in Schreier graphs of quotients of $\Gamma$ given by
   \[\mu^r(f)(E,h)=\mu(\{x: B_r(x,\ax)=E, f(\gamma\cdot x)=h([\gamma]_E)\}).\]
\end{dfn}

Once again, the density of solutions to LCLs are detected by the local statistics of general actions. But so are parameters like the girth of the Schreier graph. The $(r,k)$-local statistics and local-global containment, equivalence, and convergence are defined exactly the same for non-free actions as for free actions using the definition of $\mu^r(k)$ above.

\begin{dfn} For $\ax$ a pmp action of $\Gamma$ (not necessarily free), \[\s L_{rk}'(\ax)=\{\mu^r(f): f:X\rightarrow [k]\mbox{ measurable}\}, \;\&\quad \s L_{rk}(\ax)=\overline{\s L_{rk}'(\ax)}.\]

We say $\ax$ is local-global contained in $\bx$ if $\s L_{rk}(\ax)\subseteq \s L_{rk}(\bx)$ for all $r,k$. In this case we write $\ax\preccurlyeq_{lg} \bx$. We say $\ax$ is local-global equivalent to $\bx$ if $\s L_{rk}(\ax)=\s L_{rk}(\bx)$ for all $r,k$. And, we say $\ip{\ax_i:i\in\N}$ local-global converges to $\ax$ if $\ip{\s L_{rk}(\ax_i):i\in\N}$ converges to $\s L_{rk}(\ax)$ for all $r,k$.

\end{dfn}

To capture about the isomorphism type of the Schreier graph in the model theoretic picture, we include new constants in our language for fixed point sets.

\begin{dfn}
    For a pmp action $\ax:\Gamma\curvearrowright (X,\mu)$, $M^+(\ax)$ is $M(\ax)$ equipped with the constants $\bf 1_{\fix(\gamma)}$ for $\gamma\in \Gamma$.
    Ultraproducts of $M^+(\ax)$ are defined by taking the ultralimits of each $\bf 1_{\fix(\gamma)}$.
    The language $\s L^+$ is defined analogously to $\s L$ in definition \ref{dfn: language}, with terms and formulas defined inductively but allowing the constants $\bf 1_{\fix(\gamma)}$ as terms.

    Say that $\ax$ is $\Sigma_1^+$-equivalent to $\bx$ if $\phi^{\ax}=\phi^\bx$ for all $\Sigma_1$ sentences in $\s L^+$. In this case we write \[\ax\equiv^+_{\Sigma_1}\bx.\] We define $\Sigma_1^+$-convergence similarly.
\end{dfn}

The following is a simple adaptation of Theorem \ref{thm:freeequivalents}.

\begin{thm}\label{thm:nonfreeequivalents}
    The following are equivalent for any (not necessarily free) pmp actions of $\Gamma$, $\ax,\bx$: 
    \begin{enumerate}
        \item $\bx\preccurlyeq_w \ax$ and the factor map from $\ax^{\s U}$ preserves stabilizers (i.e.~for all $x,\gamma$, $\gamma\cdot x=x$ iff $\gamma\cdot f(x)=f(x)$).
        \item $(\forall \phi\in \Sigma^+_1)\; \phi^{M^+(\ax)}\leq \phi^{M^+(\bx)}$
        \item $\bx\preccurlyeq_{lg}\ax$.
    \end{enumerate}
    And, as before, $\ax \preccurlyeq_w \bx$ iff $(\forall \phi\in \Sigma_1) \;\phi^{\ax}\leq \phi^{\bx}$.
\end{thm}
\begin{proof}[Proof sketch]
    If $\ax^{\s U}$ factors onto $\bx$ via a stabilizer preserving map, then $M^+(\bx)$ embeds into $M^+(\ax)^{\s U}$, so any witness to a $\Sigma_1$ sentence in $M^+(\bx)$ gives a witness in $M^+(\ax)^{\s U}$. And, $M^+(\ax)$ and $M^+(\ax)^{\s U}$ are elementary equivalent.

    We can test the distribution of quotient $r$-balls in an action using the constants $\bf 1_{\fix(\gamma)}$. For instance, the measure of points where $B_r(x,\ax)\cong B_r(e)$ is
    \[\tr\bigl(\prod_{\gamma,\delta\in B_r(e)} (1-\bf 1_{\fix(\gamma\delta\inv)}) \bigr).\] We can similarly detect the statistics of labelled quotient $r$-balls in $k$-valued labellings. So, if the $\Sigma_1$ theory of $\bx$ dominates that of $\ax$, then any measure in the $(r,k)$-local statistics of $\bx$ must also be in the $(r,k)$-local statistics of $\ax$.

    If $\ax$ local-global contains $\bx$, then we can simultaneously simulate any step function in $\bx$ and any set of fixed points in $\bx$ with a function and set of fixed points in $\ax$ up to any radius $r$ and any tolerance $\epsilon$. This means that a countable dense substructure of step functions in $M^+(\bx)$ embeds into the ultrapower $M^+(\ax)^{\s U}$.
\end{proof}

And, exactly as before, we can boost $\Sigma_1$ equivalence to elementary equivalence (that's $\equiv_{\Sigma_1}$ and $\equiv$, not $\equiv_{\Sigma_1}^+$ and $\equiv^+$). We can also show that $\bf 1_{\fix(\gamma)}$ definable in $M(\ax)$, so the elementary theory of $M(\ax)$ already detects this information. These equivalence notions then all collapse.

\begin{prop}
    For any pmp actions $\ax,\bx$,
    \[\ax\equiv \bx \lra\;\ax\equiv_{\Sigma_1} b\;\lra\; a\equiv_{\Sigma_1}^+ b.\]

    Further, the $\Sigma_1$ theory of a sequence of actions converges iff the $\Sigma_1^+$ theory converges iff the elementary theory converges.
\end{prop}
\begin{proof}
    The convergence statements follow from the equivalence statements and compactness. Of course elementary equivalence and $\Sigma_1^+$ equivalence both imply $\Sigma_1$ equivalence. To show that elementary equivalence implies $\Sigma_1^+$ equivalence, it suffices to show that the fixed point projections, $\bf 1_{\fix(\gamma)}$, are all definable (in the sense that $d(f, \bf f_{\fix(\gamma)})$ is given by a first order formula). Indeed,

    \[d(f, \bf 1_{\fix(\gamma)})=\inf\{d(f,e)+d(e,\bf 1_{\fix(\gamma)}): e=e^2=e^*\},\]for any characteristic function $e$,
    \begin{align*} d(e, \bf 1_{\fix(\gamma)})^2=&\tr((e-\bf 1_{\fix(\gamma)})(e-\bf 1_{\fix(\gamma)})^*) \\ = & \tr(e+\bf 1_{\fix(\gamma)}-e\bf 1_{\fix(\gamma)})\\= & \tr(e)+\mu(\fix(\gamma))-\tr(e\bf 1_{\fix(\gamma)}),\end{align*} and by measurable coloring arguments,
    \[\mu(\fix(\gamma))=\inf\{1-\tr(e_1+e_2+e_3): e_i^2=e_i^*=e_i, e_ie_j=e_i(\gamma\cdot e_i)=0\}\] and
    \[\tr(e\bf 1_{\fix(\gamma)})= \inf\{\tr(e(1-e_1-e_2-e_3)): e_i^2=e_i^*=e_i, e_ie_j=e_i(\gamma\cdot e_i)=0\}.\]

    Lastly, $\Sigma_1$-equivalence implies weak equivalence by the previous theorem, and weak equivalence implies elementary equivalence as in Corollary \ref{cor:sigma1implieselementary}
\end{proof}

Combining the previous two theorems, we see that $\equiv_{\Sigma_1}$ and $\Sigma_1$-convergence are the same as $\equiv_{lg}$ and local-global convergence. Also, the previous theorem means that any parameter which is expressible in by a first order sentence, such as spectral radius, is continuous with respect to local-global convergence.

\subsection{Extensions to pmp graphs}

We can extend all the forgoing to pmp graphs by using the fact that all pmp graphs with degree bounded by some fixed $d$ are all induced by actions of a single finitely generated group.

\begin{dfn}
    A \textbf{pmp graph} with \textbf{invariant measure }$\mu$ is a graph $G=(V, E)$ where $(V,\mu)$ is a standard probability space and $E=\bigcup_i \iota_i$ for some list of measure preserving involutions $\ip{\iota_i: i\in\N}$. 
\end{dfn}

Note that every finite graph is a pmp graph with normalized counting measure. The basic theorems about pmp graphs can be found in \cite{KM}. For us, pmp graphs will arise as local-global limits of finite graphs and as a tool to study $\aut(G)$-fiid processes.

\begin{prop}\cite{KM}
    If $G$ is a pmp graph so that each vertex has degree at most $n$, then $G$ is the Schreier graph of some (not necessarily free) pmp action of $C_2^{*(2n+1)}$.
\end{prop}

The definition of local-global convergence for pmp graphs is due to Bollobas--Riordan in the finite case and Hatami--Lovasz--Szegedi in general \cite{BR}\cite{HLS}.

\begin{dfn}
    Let $\bb L(r, k, d)$ be the set of isomorphism classes of finite rooted graphs of degree at most $d$ with vertex and edge $[k]$-labellings (isomorphism as rooted labelled graphs). We denote elements of $\bb L(r,k,d)$ by $[(H,o,g)]$ where $H$ is a graph, $o$ is a root, and $g$ is labelling.

    For a pmp graph $\s G$ on a probability space $(X,\mu)$, and $f:X\cup  E\rightarrow [k]$ measurable\footnote{Here we give $E$ the measure defined by $\tilde\mu(A)=\int_x |A_x|d\mu=\int_x |A^x|d\mu$ for $A\subseteq E$ Borel.}, $\mu^r(f)$ is the probability measure on $\bb L(r,k,d)$ given by
    \[\mu^r(f)\bigl([H,o,g]\bigr)=\mu\bigl(\{x: (H,o,g)\cong(B^{\s G}_r(x), x, f\res (B^{\s G}_r(x)))\}\bigr).\]
\end{dfn}
Once again, the $(r,k)$-local statistics and local-global containment, equivalence, and convergence are defined exactly the same for actions, but using the definition of $\mu^r(k)$ above. We needn't include edge labellings in the definition as they can be coded in the vertex labellings, but they are convenient. See \cite{HLS}.

The following definition and proposition connect local-global convergence for graphs and local-global convergence for actions.

\begin{dfn}
    A \textbf{marking} of a pmp graph $G$ with degree bounded by $n$ is an action $C_2^{*(2n+1)}$ which induces $G$ as its Schreier graph.
\end{dfn}

Note that a given graph may have markings which are very different from the point of view of the group. For instance, $C_2^{*6}$ can act so that the first $3$ generators act freely and the last three act trivially. The associated Schreier graph has (at least) 2 markings, the original one and one where the role of different generators are swapped.

We would like to say $G$ local-global contains $H$ if and only if some marking of $G$ local-global contains a marking of $H$. Certainly, if some marking of $G$ contains a marking of $H$, then $G$ contains $H$, but the converse is a little subtler. We may have to swap out for equivalent graphs.

\begin{prop} \label{prop:localglobalequivalents}
    If $\ax$ local-global contains $\bx$, then $\sch(\ax)$ local-global contains $\sch(\bx).$ Further, if $\ip{\ax_i:i\in \N}$ local-global converges to $\ax$, then $\sch(\ax_i)$ local-global converges to $\sch(\ax)$.
\end{prop}
\begin{proof}
    The only complication is that our definition of local-global convergence for graphs includes edge labellings. But in the presence of an action, there is a canonical way of coding edge labellings\footnote{The original definition of local-global equivalence of graphs does not include edge labellings, but the authors there using a similar coding idea to draw conclusions about parameters like the matching ratio}.

    Recall $E$ is the marked generating set for $\Gamma$. Let $\sch(\bx)=(X,F)$ and $\sch(\ax)=(Y,G)$. Given an edge-labelling $f: F\rightarrow [k]$, define a vertex labelling
    \[c(f): X\rightarrow [k]^{|E|}\]
    \[c(f)(x)=\bigl(f(x,\gamma_1\cdot_a x),...,f(x,\gamma_n\cdot_a x)\bigr).\]
    So, we have a injective function $c: [k]^F\rightarrow \bigl([k]^{|E|}\bigr)^X$ with inverse $c\inv$ defined on $\im(c)$.
    
    Note that, fixing a bijection between $[k^{|E|}]$ and $[k]^{|E|}$, for any measure $\mu$ on vertex and edge labelled radius $r$ graph
    \[\phi_\mu(f)=d_{TV}\bigl(\mu^{r}(c\inv(f)),\mu\bigr) \] is given by a $\Sigma_1$ sentence in $M^+(\ax)$. 

    So, if $\ax$ local-global contains $\bx$, then for any edge-and-vertex labelling $\mu\in \s L_{rk}(H)$ there is $f\in M^+(\bx)$ with $\phi^{b}_\mu(f)<\epsilon$, and so there is $f'\in M^+(\ax)$ with $\phi^{\ax}_\mu(f)<\epsilon$ and $\mu\in \s L_{rk}(\ax)$.

    Further, if $\ax_i$ is a local-global convergent sequence of actions, then $\inf_{f\in M^+(\ax_i)} \phi^{\ax_i}_{\mu}(f)$ converges for each $\mu$, so $\sch(\ax_i)$ local-global converges.
\end{proof}

This coding can all be done with higher-arity labellings as well, though we won't need this in what follows. As a partial converse:

\begin{prop}
    For pmp graphs $G$ and $H$, $H\preccurlyeq_{lg} G$ iff, for any marking $\bx$ of $H$, there an action $\ax$ so that $\sch(\ax)\equiv_{lg} G$ and $\bx \preccurlyeq_{lg} \ax$.
\end{prop}
\begin{proof}
    Say $H$ has vertex set $Y$ and $G$ has vertex set $X$, and fix $H$ and a marking $\bx:\Gamma\curvearrowright Y$ of $H$. The previous theorem implies that if $G$ is local-global equivalent to the Schreier graph of some action that weakly contains $\bx$, then $G$ local-global contains $H$.
    
    For the converse, the basic idea is that a marking is coded by a certain kind of edge labelling, and local-global containment lets us approximate the edge labelling associated to $\bx$ on $G$. Then, we can patch together these approximations in an ultraproduct of (some marking of) $G$.

    The diagram that codes an action $\ax:\Gamma\curvearrowright X$ is defined as follows. For an edge $(x,y)$ in $\sch(\ax)$, let 
    \[d_{\ax}(x,y)=\{\gamma\in E: \gamma\cdot x=y\}.\] (So , this labelling takes values in $\s P(E)$.)
    
    I claim that $\bx\preccurlyeq_{lg}\ax$ if and only if, for every $g: Y\rightarrow [k]$, $r\in \N$, and $\epsilon>0$, we can find $\tilde g: X\rightarrow [k]$ so that
    \[d_{TV}\bigl(\mu^r(g,d_{\bx}),\mu^r(\tilde g, d_{\ax})\bigr)<\epsilon.\] Here, by $\mu^r(g,d_{\bx})$ we mean the joint $r$-local statistics of $g$ and $d_\bx$.  This is because, for a labelled quotient $r$-ball $(E,h)$,
    \begin{align*}\m^r(g)(E,h)&=\bb P\bigl((\forall \gamma,\delta\in B_{r}(e))\;\delta\cdot x=\gamma\cdot x \leftrightarrow E(\delta,\gamma)\wedge g(\gamma\cdot x)=h(\gamma)\bigr)\\
    &= \mu^r\bigl([B_r(e)/E, [e]_E, (h,d_{st})] \bigr)
    \end{align*} where $d_{st}$ is the standard diagram given by $d_{st}([x],[y])=\{\gamma: [\gamma\cdot x]=[y]\}$.
 
    Fix any marking $\ax_0$ of $G$. Since $H\preccurlyeq_{lg} G$, for any $\epsilon>0, r\in\N,$ and $g: Y\rightarrow [k]$, we can find $d,\tilde g$ so that
    \[d_{TV}\bigl(\mu^r(g, d_{\bx}),\mu^r(\tilde g, d)\bigr)<\epsilon.\] By coding edge-labellings as above, we can use saturation to find a single edge labelling $d\in M(\ax_0)^{\s U}$ which works for all $g, r,$ and $ \epsilon$ simultaneously. (Or at least for all $g$ in a countable dense set). By Theorem \ref{thm:lowenheimskolem}, there is some $\ax_1:\Gamma\curvearrowright Z$ and an elementary embedding $e: M(\ax_1)\preccurlyeq M(\ax_0)^{\s U}$ so that $d\in\im(e)$. We have, $\ax_1\equiv \ax_0$, so by the first direction of $\sch(\ax_1)\equiv_{lg} \sch(\ax_0)=G$.
    
    Define an action $\ax:\Gamma\curvearrowright Z$ by setting $\gamma\cdot_{\ax} x=y$ if and only if $\gamma\in d(x,y)$. Then $\ax$ local-global contains $\bx$ by construction and $\sch(\ax)=\sch(\ax_1)\equiv_{lg} G$.
\end{proof}

Unfortunately, we do need to swap out $G$ in some cases. For instance, all graphs which are locally isomorphic to $\Z$ are local-global equivalent, but some of these graphs don't come from $\Z$ actions while others do. 

From the compactness, saturation, and continuity theorems for actions, we can conclude analogous theorems for graphs.

\begin{prop}
    If $\ip{G_i:i\in\N}$ is a sequence of pmp graphs all of bounded degree, and there is some sequence of markings $\ip{\ax_i:i\in\N}$ which elementary converges, then $\ip{G_i:i\in\N}$ local-global converges.
\end{prop}
\begin{proof}
    This is just the proposition above and Theorem \ref{thm:nonfreeequivalents} on containment for nonfree actions.
\end{proof}
\begin{cor}[Hatami--Lovasz--Szegedi]
    If $\ip{G_i:i\in\N}$ is a local-global Cauchy sequence of pmp graphs then $\ip{G_i:i\in\N}$ local-global converges.
\end{cor}
\begin{proof}
    It suffices to find a limiting class for some subsequence. Let $\ip{\ax_i:i\in\N}$ be a sequence of markings of $G_i$. By model theoretic compactness, there is an elementary convergent subsequence of $\ip{\ax_i:i\in\N}$. By the above, the Schreier graph of this limit is a local-global limit for $\ip{G_i:i\in\N}$.
\end{proof}
This proof is rather different than the original proof in \cite{HLS} (which was stated only for sequences of finite graphs). And, it plays out a recurring theme: though there is not an obvious or coherent way to choose markings for pmp graphs, compactness totally obviates the issue.

\begin{prop}\label{prop:replace}
For any bounded degree graph $G$ and any $\mu\in \s L_{rk}(G)$ there is some $H$ local-global equivalent to $G$ so that $\mu=\mu^r(f)$ for a measurable labelling $f$ of $H$.
\end{prop}

\begin{proof}
    Let $M$ be the ultrapower of any marking of $G$. Then, $\phi^\mu(f)=0$ for some $f\in M$. Let $H$ be the Schreier graph of any standard actions with elementary embedding into $M$ that contains $f$.
\end{proof}

For example, if $G$ has an approximate 2-coloring it is local-global equivalent to a graph with a 2-coloring.

\begin{prop}
    Any graph-theoretic parameter expressible by a first order sentence in $M(\ax)^+$ is continuous with respect to local-global convergence.
\end{prop}
So, for instance, the spectral radius of the adjacency operator is continuous with respect to local-global convergence.

\section{Cayley diagrams and lifting lemmas}

Fix a marked group $\Gamma$ with generating set $E$. In this section, we'll characterize when $\Gamma$-fiid measures and $\aut(\cay(\Gamma))$-fiid measures have the same combinatorics. First, we'll settle some conventions about actions and show how questions about factor of iid measures translate into problems about pmp graphs. 

\subsection{Fiid processes and Bernoulli shifts}
 Recall that we define $\cay(\Gamma)$ as the Schreier graph with respect to the left action of $\Gamma$ on itself. The edges in $\cay (\Gamma)$ are of the form $(\gamma\cdot x, x)$ with $x\in \Gamma$ and $\gamma$ in the generating set $E$. So, $\Gamma$ embeds into $\aut(\cay(\Gamma))$ by $\gamma\mapsto \ip{\gamma}$ where $\ip{\gamma}(x)=x\gamma\inv$. If the context is clear we may suppress these wickets. This also decides our convention for Bernoulli shifts.

 \begin{dfn}
     For $A$ any set of labels and $X$ a space of (edge or vertex) labellings of $\cay(\Gamma)$, the shift action $s:\aut(\cay(\Gamma))\curvearrowright X$ is given by shifting indices. That is,
     \[f\cdot x=x\circ f\inv\]

     In particular, $s:\Gamma\curvearrowright A^\Gamma$ is given by
     \[(\gamma\cdot x)(\delta)=(\ip{\gamma}\cdot x)(\delta)=x\bigl(\ip{\gamma}\inv (\delta)\bigr)=x(\delta\gamma).\]
 \end{dfn}

We can naturally associate a pmp graph to the action of $\Gamma$ on a Bernoulli shift, namely the Schreier graph. The labellings of this graph exactly correspond to $\Gamma$-factor of iid processes. It turns out there is a similar pmp graph whose labellings correspond to $\aut(\cay(\Gamma))$-fiid processes.

\begin{dfn}
    For a finitely generated, marked group $\Gamma$,
    \[\Fr([0,1]^{\Gamma})=\bigl\{x\in [0,1]^\Gamma: (\forall f\in \aut(\cay(\Gamma)))\;\bigl(f\not=\operatorname{id} \rightarrow f\cdot x\not= x\bigr)\bigr\}.\] Note that $\Fr([0,1]^{\Gamma})$ has measure $1$ with respect to $\lambda^\Gamma$, where $\lambda$ is Lebesgue measure (it contains all the injective labellings).

    The \textbf{Bernoulli shift graph} of $\Gamma$ is
    \[\s S(\Gamma)=\sch(s)\] where $s: \Gamma\curvearrowright \Fr([0,1])^{\Gamma}$ is the shift action above. That is, $\s S(\Gamma)$ has vertex set $[0,1]^\Gamma$ and edges $\{x,\gamma\cdot x\}$ for $x\in [0,1]^\Gamma$ and $\gamma\in E$.

    Let $R=\aut_e\bigl(\cay(\Gamma)\bigr)=\{f\in \aut(\cay(\Gamma)): f(e)=e\}$. The \textbf{Bernoulli graphing} of $\cay(\Gamma)$ is the graph $\widetilde{\s S}(\cay(\Gamma))$ on $\Fr([0,1]^{\Gamma})/R$ with edges
    \[\{(R\cdot x,R\cdot y):(\exists \gamma\in E)\;  R\gamma \cdot x\cap R\cdot y\not=\emptyset\}.\]
\end{dfn}

This is a special case of a more general construction of Bernoulli graphings over unimodular random graphs. One important point is that $\widetilde{\s S}(\cay(\Gamma))$ really is a function of the graph $\cay(\Gamma)$. The proof below gives a ``coordiate free" description of $\s S(\cay(\Gamma))$.

\begin{prop}
If $G, H$ are two Cayley graphs (possibly of different groups)  and $G\cong H$, then there is a measure-preserving graph isomorphism between $\widetilde{\s S}(G)$ and $\widetilde{\s S}(H)$
\end{prop}
\begin{proof}
    Fix a root for $G$, say $*$. We can equivalently present $\widetilde{\s S}(G)$ as the graph whose vertex set is injective labellings of $(G, *)$ up to root-preserving isomorphism where two labellings are adjacent if one can be gotten from the other by moving the root. Clearly an isomorphism between $G$ and $H$ induces an isomorphism between these graphs.
\end{proof}

Conley, Kechris, and Tucker-Drob show that $\Fr([0,1]^{\Gamma})/R$ is a standard Borel space, that every component of $\widetilde{\s S}(\cay(\Gamma))$ is isomorphic to $\cay(\Gamma)$, and that $\widetilde{\s S}(\cay(\Gamma))$ preserves the quotient measure and is expansive when $\Gamma$ is nonamenable. For completeness we include a brief sketch of their argument. Note that what we call $\widetilde {\s S}(\cay(\Gamma))$ is called $E$ in \cite{CKTD}. In several places, we will rely on the following computation:

\begin{prop} \label{prop:factorizations} Every element of $\aut(\cay(\Gamma))$ factorizes uniquely as $\ip{\gamma}\,r$ or $r'\ip{\gamma'}$ where $r,r'\in \aut_e(G)$ and $\gamma,\gamma'\in\Gamma$. These factorizations are related as follows:
\[r\,\ip{\gamma}=\ip{r(\gamma\inv)\inv} \bigl( \ip{r(\gamma\inv)}\, r \,\ip{\gamma}\bigr)\] 
\[\ip{\gamma} \,r=\bigl(\ip{\gamma}\,r\,\ip{r\inv(\gamma)\inv}\bigr)\ip{r\inv(\gamma)}\] with $ \bigl( \ip{r(\gamma\inv)}\, r \,\ip{\gamma}\bigr)$ and $\bigl(\ip{\gamma}\,r\,\ip{r\inv(\gamma)\inv}\bigr)$ $\aut_e(G)$.
\end{prop}
These expressions are clunky, but for the most part we only need to know that $\gamma'$ is in $E$ if and only if $\gamma$ is in $E$ when $\gamma,\gamma'$ are in the left and right factors of an automorphism as above.

\begin{prop}[{\cite[Lemma 7.9]{CKTD}}] \label{gtilde}
For any $x\in\Fr\bigl([0,1]^\Gamma\bigr)$, the map $\gamma\mapsto (R\gamma\inv)\cdot x$ is an isomorphism between $\cay(\Gamma)$ and the component of $R\cdot x$ in $\widetilde{ \s S}(\Gamma, E)$.
\end{prop}
\begin{proof}[Sketch of proof] In fact, we'll prove something stronger. Let $\pi: \Fr([0,1]^\Gamma)\rightarrow \Fr([0,1]^\Gamma)/R$ be the quotient map
\[\pi(x)=R\cdot x.\] We'll show that for each $x\in\widetilde{\s S}(\cay (\Gamma))$, $\pi\inv[x]$ meets each component of $\s S(\Gamma)$ at most once, and that an edge between $x$ and $y$ in $\widetilde{\s S}(\cay(\Gamma))$ corresponds to a perfect matching between $\pi\inv[x]$ and $\pi\inv[y]$ in $\s S(\Gamma)$. Then, $\pi$ restricts to an isomorphism on components of $\s S(\Gamma)$ and the proposition follows.

Consider a vertex $R\cdot x$ of $\widetilde{\s S}(\cay(\Gamma))$ and suppose that $x'\in (R\cdot x)\cap (\Gamma\cdot x)$ is another point in $R\cdot x$ in the same component of $\s S(\Gamma)$ as $x$. This means $r\cdot x=\gamma\cdot x$ for some $r\in R$ and $\gamma\in \Gamma$. By free-ness of the action, $\gamma=r$. But since $r$ doesn't move the identity, $\gamma=r=e$. So each vertex of $\widetilde{\s S}(\cay(\Gamma))$ meets each component of $\s S(\Gamma)$ at most once.

Now, suppose that $R\cdot x$ and $R\cdot y$ are neighbors in $\widetilde {\s S}(\cay (\Gamma))$. Then there is some $r\in R$ and $\gamma\in E$ so that $r\inv\cdot y=\gamma\cdot x$. So, by Proposition \ref{prop:factorizations}, \[y=r\,\ip{\gamma}\cdot x=\ip{r(\gamma\inv)\inv}\,r'\cdot x\] for some $r'\in x$. Since $E$ is the set of vertices of $\cay{\Gamma}$ at distance $1$ from the identity, $r(\gamma\inv)\inv\in E$. This means $(y, \ip{r(\gamma\inv)\inv}\,r' \cdot x)$ is an edge in $\s S(\Gamma)$. We hae shown that every element of $R\cdot y$ is adjacent to some element of $R\cdot x$. By symmetry, the reverse is true. And so, by the previous paragraph, $\s S(\Gamma)$ restricted to $R\cdot x$ and $R\cdot y$ is a perfect matching.
\end{proof}

The other properties of $\widetilde {\s S}(\Gamma, E)$ in \cite{CKTD} follow similarly by considering edges in $\widetilde {\s S}(\Gamma, E)$ as matchings in $\s S(\Gamma, E).$

Any of our definitions about local statistics can also be made for random labellings of bounded degree graphs. But, the only such random labellings which figure in this paper are factor of iid processes. We won't burden the reader with yet more definitions, and will instead rely on the correspondence between fiid labellings of $\cay(\Gamma)$ and measurable labellings of $\s S(\Gamma)$ and $\widetilde {\s S}(\cay(\Gamma))$.

The following proposition makes this correspondence between fiid processes and pmp graphs explicit:

\begin{prop}[c.f. the proof of {\cite[theorem 7.7]{CKTD}}]\label{meastoiid} Fix a marked group $\Gamma$ and let $G=\cay(\Gamma)$. The following are equivalent for any $\nu\in P_1\bigl(\bb L(r,k,\Gamma)\bigr)$: 
\begin{enumerate}
\item There is a $\Gamma$-fiid random $[k]$-labelling of $G$, $\bf f$, so that for any $f\in \bb L(r,k,\Gamma)$
\[\bb P(\bf f\res B_r(e)=f)=\nu(\{f\})\]
\item There is a measurable labelling $f$ of $\s S(\Gamma)$ with $\mu^r(f)=\nu$.
\end{enumerate}
Similarly, for $\nu\in P_1\bigl(\bb L(r,k,d)\bigr)$, the following are equivalent:
\begin{enumerate}
    \item There is an $\aut(G)$-fiid random labelling of $G$, $\bf f$, so that, for any possible colored finite neighborhood $[H,o,f]\in \bb L(r,k,d)$
    \[\nu\bigl(\{[H,o,f]\}\bigr)=\bb P\bigl( (H,o,f)\approx(B_r(e), e,\bf f\res B_r(e))\bigr)\]
    \item There is a measurable labelling $f$ of $\widetilde{\s S}(G)$ with $\mu^r(f)=\nu$.
\end{enumerate}
\end{prop}
\begin{proof}

For the first statement, suppose $f$ is a measurable $[k]$-labelling of $\s S(\Gamma)$. We'll define a $\Gamma$-fiid random labelling $\bf f$ of $\Gamma$ with the same local statistics. For $x\in \Fr([0,1]^\Gamma)$ set \begin{equation*}\label{eq:msrtoiidalg} \bf f(x)(\gamma)=f(\gamma\cdot x).\end{equation*} Then, 
\begin{align*} \bf f(\delta\cdot x)(\gamma) & = f(\gamma \delta\cdot x) \\
 & = \bf f(x)(\gamma\delta) \\
 \; & = \bigl(\delta\cdot\bf f(x)\bigr)(\gamma),
\end{align*} so $\bf f$ is equivariant. Now we compute the local statistics. Write $\mu$ for $\lambda^\Gamma$, the measure preserved by $\s S(\Gamma)$. For any $h: B_r(e)\rightarrow [k]$,
\[\bb P\bigl((\bf f(x)\res B_r(e))=h\bigr)=\mu\bigl(\{x: (\forall \gamma\in B_r(e))\;f(\gamma\cdot x)=h(\gamma)\}\bigr) =\mu^r(f). \]

Conversely, if we're given a factor map $\bf f$, the same algebra shows we can define a measurable labelling $f$ via $f(x)=\bf f(x)(e)$.


For the second statement, we'll handle the case of vertex labellings. The general case is only a notational change. Abbreviate $\aut_e(G)$ as $R$. If $\widetilde f$ is a measurable $[k]$-labelling of $\widetilde{ \s S}(\Gamma, E)$, then Proposition \ref{gtilde} implies that $f(x)=\widetilde f(R\cdot x)$ is a measurable labelling of $\s S(\Gamma)$ which is invariant under the action of $R$. As above, this gives rise to a $\Gamma$-factor map $\bf f$ with the same local statistics. 

Since $\aut(G)=\Gamma R$, it suffices to check that $\bf f$ is also $R$-equivariant. By the factorization relations, Proposition \ref{prop:factorizations} above, for any $\gamma\in \Gamma,$ $\ip{\gamma}\, r=r'\,\ip{r\inv(\gamma)}$ for some $r'$. So \begin{align*}\bf f(r\cdot x)(\gamma) & = f((\ip{\gamma}\, r)\cdot x) \\
 & = f(r'\ip{r\inv(\gamma)}\cdot x) \\
 & = f(\ip{r\inv(\gamma)} \cdot x) \\
 & = \bf f(x)(r\inv(\gamma)) \\
 & = (r\cdot \bf f(x))(\gamma).
\end{align*}

Conversely, if $\bf f$ is an $\aut(G)$-\fiid labelling of $G$, then as above $\bf f$ gives rise to an $R$-invariant measurable labelling of $\s S(\Gamma)$, $f$. Since $f$ is $R$-equivariant, $\widetilde f(R\cdot x)=f(x)$ defines a measurable $\varphi$-decoration of $\widetilde{ \s S}(\Gamma)$.
\end{proof} 

\subsection{Cayley diagrams}

Earlier we used the fact that actions of $C_2^{*n}$ are coded by edge labellings with simple combinatorial constraints. One can check that free actions are coded by proper edge colorings. Cayley diagrams generalize this idea.

\begin{dfn}
    A $\Gamma$-\textbf{Cayley diagram} on a graph $G=(V,F)$ is a labelling of edges by generators, $c: F\rightarrow E$, so that
    \begin{enumerate}
    \item for every vertex $x$ and every $\gamma\in E$ there is at exactly one neighbour with $c(x,y)=\gamma$
    \item for any path $(x_0,...,x_n)$ in $G$, $c(x_n,x_{n-1})...c(x_{1},x_0)=e$ if and only if $x_0=x_n.$
    \end{enumerate}

    The \textbf{standard Cayley diagram} of $\cay(\Gamma)$ is
    \[c_{st}(\gamma x,x)=\gamma.\]
\end{dfn}

A graph can only admit a $\Gamma$-Cayley diagram if every component is isomorphic to $\cay(\Gamma)$, and whether an edge labelling $d$ is a Cayley diagram is determined by the sequence of measures $\mu^r(d)$ for $r\in \N$. In fact, if $\Gamma$ is finitely presented with no relation longer than $n$, then $\mu^{n+1}$ detects whether $d$ is a Cayley diagram. The next proposition is key:

\begin{prop}\label{prop:cdaction}
    If $G$ is a (pmp) graph which admits a (measurable) $\Gamma$-Cayley diagram, then $G=\sch(\ax)$ for a free (pmp) action of $\Gamma$.
\end{prop}
\begin{proof}
    Suppose $c:G\rightarrow E$ is a Cayley diagram and define an action on generators by setting $\gamma\cdot x=y$ if $c(y,x)=\gamma$. This is well defined by point $(1)$ of the definition of Cayley diagrams and it extends to a free action by part $(2)$.
\end{proof}

Cayley diagrams come with quite a lot of structure. The next few propositions give an exact correspondence between Cayley diagrams and elements of $\aut_e\bigl(\cay(\Gamma)\bigr)$.

\begin{prop}\label{prop:actionrotationdiagram}
For any Cayley diagram $c$ on $G=\cay(\Gamma)$, there is a unique rotation $r_c\in\aut_e(G)$ and action $\cdot_c: \Gamma\curvearrowright G$ so that for all $x\in \Gamma, \gamma\in E$ the following diagram commutes:
\[
\xymatrix{
x \ar[r]^{r_c} \ar[d]^{\gamma} & r_c(x) \ar[d]^{c(\gamma x,x)} \\
\gamma x\ar[r]^{r_c} &\gamma\cdot_c r_c(x)
} 
\]
meaning
\[r_c(\gamma x)=c(\gamma x, x)r_c(x)=\gamma\cdot_c r_c(x).\]
Similarly, any rotation or action gives a unique way to fill out this diagram.
\end{prop}
\begin{proof}
Uniqueness is clear since $r_c(e)=e$ and the definition tells us how to compute $r_c(\gamma)$ for $\gamma=\gamma_n...\gamma_2\gamma_1$ by filling out the diagram \[e\rightarrow \gamma_1\rightarrow \gamma_2\gamma_1\rightarrow ... \rightarrow \gamma.\] For existence, it suffices to note that item $(2)$ in the definition of Cayley diagram tells us that any two expressions for $\gamma$ in terms of generators gives the same value of $r_c(\gamma)$ in this computation.
\end{proof}

It will be useful to know how this translation interacts with automorphisms.

\begin{dfn}
    $\mathrm{CD}$ is the space of Cayley diagrams of $\cay(\Gamma)$.
\end{dfn}

\begin{prop}
Let $\aut_e(G)$ act on $\cd$ by shifting indices. For $r\in \aut_e(G)$, $\gamma\in \Gamma$ and $d\in \cd$, we have
\[r_{r\cdot c}=r_c\circ r\inv \mbox{, and}\quad r_{\ip{\gamma}\cdot c}=\ip{r_c(\gamma)}\,r_c\,\ip{\gamma}\inv \]
\end{prop}
\begin{proof}
Let $h=\ip{r_c(\gamma)}\,r_c\ip{\gamma}\inv$. We have that, for any $x\in \Gamma$, $\delta\in E$, 

\[(r\cdot c)(\delta x, x)=c(r\inv(x), r\inv(\delta x))\] and 
\[(\ip{\gamma}\cdot c)(\delta x, x)=c(\ip{\gamma}\inv(\delta x), \ip{\gamma}\inv(x))=c(\delta x\gamma, x\gamma).\] And, $h(x)=r_c(x\gamma)r_c(\gamma)\inv$. So the following diagrams commute:
\[
\xymatrix{
x \ar[r]^<<<<<<{r_c r\inv} \ar[d]^{\delta} & (r_c r\inv)(x) \ar[d]^{(r\cdot c)(\delta x, x)} &\quad  & x \ar[r]^<<<<<{h} \ar[d]^{\delta} & r_c( x\gamma)r_c(\gamma)\inv \ar[d]^{(\ip{\gamma }\cdot c)(\delta x, x)}  \\
\delta x\ar[r]^<<<<<<{r_c r\inv} &(r_c r\inv)(x) & & \delta x\ar[r]^<<<<<{h} & r_c(\delta x\gamma)r_c(\gamma)\inv
} 
\] The proposition then follows by the uniqueness of $r_c$.
\end{proof}
\begin{cor}
The action of $\aut_e(G)$ on $\cd$ is free and transitive. Thus, $\cd$ admits a unique $\aut(G)$-invariant measure.
\end{cor}
\begin{proof}
From the computation of $r_{r\cdot c}$ we can see that, for any $d_0\in \cd$, the map $r\mapsto r\cdot c_0$ has inverse $d\mapsto r_c\inv r_{d_0}$. Since the group $R:=\aut_e(G)$ is compact it has a unique 2-sided invariant Haar measure, $h$. And since the diagram $d_0(x,\gamma x)=\gamma$ is invariant under the action of $\Gamma$, we get an invariant measure
\[\mu(A)=\int_{r\in R} \mathbf{1}_{r\cdot c_e\in A}\;dh.\] Conversely, any invariant measure $\mu$ on $\cd$ gives an invariant measure $\tilde h$ on $R$ by \[\tilde h(A)=\mu(A\cdot c_0)\] which must agree with $h$.
\end{proof}

Note that this last corollary means that checking if there is an $\aut(G)$-fiid Cayley diagram amounts to checking if this unique measure $\mu$ is $\aut(G)$-fiid. This puts the question squarely in the realm of Ornstein theory, i.e.~the general problem of determining which measures are fiid.

\subsection{Lifting lemmas}

If $\widetilde{\s S}(G)$ has a measurable Cayley diagram, then any $\Gamma$-fiid labelling of $G$ can be lifted to an $\aut(G)$-fiid labelling. This is essentially folklore (see the comments after \cite[Question 2.4]{Lyons2016FactorsOI}), but we record a proof here for completeness. The basic idea is to average over the action of $\aut_e(G)$, this time using our Cayley diagram to choose a random rotation and resample our variables.

\begin{prop}
    If $G=\cay(\Gamma)$ admits an $\aut(G)$-fiid Cayley diagram, then for any $\Gamma$-fiid random labelling $\bf f: G\rightarrow [k]$, there is an $\aut(G)$-fiid random labelling $\tilde{\bf f}$ so that for any $[H,o,g]\in \bb L(r,k,d)$
    \[\bb P\bigl((B_r(e),e,\bf f)\approx (H,o,g)\bigr)=\bb P\bigl((B_r(e),e, \tilde{\bf f})\approx (H,o,g)\bigr).\]
\end{prop} 

\begin{proof}
For clarity we will distinguish between $\gamma\in \Gamma$ and its canonical image in $\aut(G)$, $\ip{\gamma}$.

Let $F$ be an $\aut(G)$-factor map from $[0,1]^\Gamma$ to $\cd$, and let $\bf f$ be a $\Gamma$-fiid random labelling of $G$. For the sake of readability we will assume $\bf f$ is a vertex labelling, but the general case doesn't introduce any conceptual difficulty.

We can associate to $x\in [0,1]^\Gamma$ an element of $\aut_e(G)$ given by $h_x:=r_{F(x)}$, where $r_c$ is as in Proposition \ref{prop:actionrotationdiagram}. From the previous proposition $h_{r\cdot x}=h_xr\inv$ and $h_{\gamma\cdot x}=\ip{r_c(\gamma)}\,r_c\, \ip{\gamma}\inv$

We will build a factor map out of $([0,1]^2)^\Gamma$ (this is of course isomorphic to $[0,1]^\Gamma$). Set
\[ \widetilde{\bf f}(x,y)= h_{x}\inv\cdot\bf f(h_{x}\cdot y).\] We can check that $\widetilde{\bf f}$ and $\bf f$ have the same local statistics. Since $h_y$ acts on $[0,1]^\Gamma$ in a measure preserving way, for any isomorphism invariant set of labellings $A$,
\begin{align*}
    \bb P_{xy}\bigl( \tilde{\bf f}(x,y)\res B_r(e)\in A\bigr) = & \bb P_{xy}\biggl( \bigl(h_{y}\inv \cdot \bf f(h_y \cdot x)\bigr)\res B_r(e) \in A\biggr)\\ = & \bb P_{x,y} \bigl( \bf f(h_y\cdot x)\res B_r(e)\in A \bigr) \\
    = & \bb \int_y\bb  P_{x}\bigl(\bf f(h_y\cdot x)\res B_r(e)\in A\bigr) \;d\lambda^{G}\\
    = & \int_y \bb P_x\bigl( \bf f(x) \res B_r(e)\in A\bigr) \;d\lambda^G\\
    = & \bb P_x \bigl(\bf f(x)\res B_r(e)\in A\bigr)
\end{align*}

Now we just need to show that $\widetilde {\bf f}$ is $\aut(G)$-equivariant. We check against automorphisms $\ip{\gamma}$ and $r\in \aut_e(G)$ separately. If $\gamma\in \Gamma$, 
\begin{align*}\widetilde {\bf f}(\ip{\gamma}\cdot(x,y)) & =h_{\ip{\gamma}\cdot x}\inv\cdot \bf f(h_{\ip{\gamma}\cdot x}\ip{\gamma}\cdot y) \\
  &= \left(\ip{h_{x}(\gamma)} h_{x} \ip{\gamma}\inv\right)\inv\cdot  \bf f( \ip{h_{x}(\gamma)} h_{x} \ip{\gamma}\inv\ip{\gamma}\cdot y) \\
  &= \ip{\gamma} h_{x}\inv \cdot \bf f(h_{x}\cdot y) \\
  & = \ip{\gamma} \cdot \widetilde{\bf f}(x,y)\end{align*}
  and
  \[h_{r\cdot x}\inv \cdot \bf f(h_{r\cdot x}r\cdot y )= r h_{x}\inv\cdot \bf f(h_{x}\cdot y).\]
\end{proof}

This proposition above tells us that if $\widetilde{\s S}(\cay(\Gamma))$ admits a Cayley diagram, then every $\Gamma$-fiid labelling of $\Gamma$ has \emph{exactly} the same local statistics as an $\aut(G)$-fiid labelling. If we allow ourselves any $\epsilon$ of slack, that is if we want to characterize $\s L_{rk}(\widetilde{\s S}(G))$, then it's enough to know that we can approximate a Cayley diagram.

\begin{dfn}
    An \textbf{approximate Cayley diagram} for a pmp graph $H$ is a sequence of edge labellings $\ip{f_i: i\in\N}$ so that the set of vertices where one of the defining properties of Cayley diagrams is violated vanishes. That is, for all $r$
    \[\mu^r(f_i)\bigl([B_r(e),e,c_{st}]\bigr)\rightarrow 1.\] 
\end{dfn}

Note that, by Proposition \ref{prop:replace} a pmp graph $H$ has an approximate Cayley diagram if and only if it is local-global equivalent to a pmp graph with a measurable Caley diagram.

\begin{thm}
    If $G=\cay(\Gamma)$ admits an $\aut(G)$-fiid Cayley diagram then $\widetilde{\s S}(G)$ is local-global equivalent to $\s S(\Gamma)$. In terms of fiid labellings, for any $\Gamma$-fiid random labelling $\bf f: G\rightarrow [k]$, and any $[H,o,g]\in \bb L(r,k,d)$,
    \[\bb P\bigl((H,o,g)\approx (B_r(e),e,\bf f)\bigr)=\lim_{i\rightarrow \infty} \bb P\bigl((H,o,g)\approx (B_r(e),e,\bf f_i)\bigr)\] for some sequence of $\aut(G)$-fiid labellings $\ip{\bf f_i:i\in\N}$.
\end{thm}
\begin{proof}
    In this case, $\widetilde{\s S}(G)$ is local-global equivalent to a pmp graph $H$ with a measurable Cayley diagram. Since $H$ has a measurable Cayley diagram, it is the Schreier graph of of some free pmp action of $\Gamma$, say $\ax$. By \'Abert--Weiss, $\ax$ weakly contains the shift action, $s$. By Proposition \ref{prop:localglobalequivalents}, $H$ local-global contains $\s S(\Gamma)$.
\end{proof}

\begin{cor}
    If $G=\cay(\Gamma)$ has an $\aut(G)$-fiid approximate Cayley diagram, then $\s S(\Gamma)$ and $\widetilde{\s S}(G)$ have the same independence ratio, matching ratio, etc.
\end{cor}

There are more hands-on ways to prove this theorem, of course, but the limit theory neatly packages the calculations involved.

\section{Examples}

So, which groups have approximate and measurable Cayley diagrams? First off, we should isolate the trivial case where $\Gamma=\aut(G)$.

\begin{dfn}
    A graph $G$ is a \textbf{graphical regular representation}, or \textbf{GRR}, of $\Gamma$ if $\Gamma$ acts freely and transtitively on $G$ by automorphisms. 
\end{dfn}

A GRR of $\Gamma$ must be a Cayley graph of $\Gamma$. For instance, For instance $D_\infty={\ip{ r,s: rs=sr\inv, s^2=e}}$ has a GRR induced by the generating set $\{r,s,rs,r^{-2}s \}$ (see Figure \ref{fig:grr}), and $F_2=\ip{a,b}$ has a GRR induced by $\{a,b,ab,ba,baa\}$. Of course if $\Gamma$ is a marked group whose Cayley graph $G$ is a GRR, $\widetilde{\s S}(G)$ has an $\aut(G)$-fiid Cayley diagram.

\begin{figure}
\centering
\begin{tikzpicture}

\draw[black, thin] (-5,1) -- (5,1);
\draw[black, thin] (-5,-1) -- (5,-1);
\draw[black, thin] (-4,1) -- (-4,-1);
\draw[black, thin] (-2,1) -- (-2,-1);
\draw[black, thin] (0,1) -- (0,-1);
\draw[black, thin] (2,1) -- (2,-1);
\draw[black, thin] (4,1) -- (4,-1);
\draw[black, thin] (-4,1) -- (-2,-1);
\draw[black, thin] (-2,1) -- (0,-1);
\draw[black, thin] (0,1) -- (2,-1);
\draw[black, thin] (2,1) -- (4,-1);
\filldraw[black] (-4,1) circle (2pt);
\filldraw[black] (-2,1) circle (2pt);
\filldraw[black] (0,1) circle (2pt);
\filldraw[black] (2,1) circle (2pt);
\filldraw[black] (4,1) circle (2pt);
\filldraw[black] (-4,-1) circle (2pt);
\filldraw[black] (-2,-1) circle (2pt);
\filldraw[black] (0,-1) circle (2pt);
\filldraw[black] (2,-1) circle (2pt);
\filldraw[black] (4,-1) circle (2pt);

\foreach \x in {-4, -2, 0}{
    \draw[black, thin] (\x,-1) -- (\x+4,1);}

\end{tikzpicture}
\caption{ A GRR for $D_\infty$} \label{fig:grr}
\end{figure}

 Slightly less trivially, actions of amenable groups are measure hyperfinite, so their approximate theory collapses to the classical combinatorics of their Cayley graph. A straightforward argument gives amenability of the connectedness relation of Bernoulli graphings over the Cayley graphs of amenable group, so these are local-global equivalent to the corresponding Bernoulli shift graphs. In detail:

\begin{thm}\label{approxamen}
If $\Gamma$ is an amenable marked group and $G=\cay(\Gamma)$, then 
\begin{enumerate}
    \item All free pmp actions of $\Gamma$ are weakly equivalent
    \item If every component of $H$ is isomorphic to $G$, then $H$ is local-global equivalent to $\s S(\Gamma)$.
\end{enumerate}
\end{thm}
\begin{proof}
Point $(1)$ is well known and can be found in \cite[Prop 13.2]{GlobalAspects}. Roughly, hyperfiniteness of the action of the action allows you to approximate any labelling on large almost invariant sets. For point $(2)$ we show that $\widetilde{\s S}(G)$ is hyperfinite, then use this same idea to get an approximate Cayley diagram. The rest follows by the Hatami--Lovasz--Szegedy containment theorem.

To get hyperfiniteness we will show that the connectedness relation of $\widetilde{ \s S}(\Gamma, E)$ is amenable (as in {\cite[Chapter 9]{KM}}). Since $\Gamma$ is amenable, we have amenability measures $\nu_x$ for the action of $\Gamma$ on $\Fr([0,1]^G)$. Abbreviate $\aut_e(G)$ as $R$. Define
\[\nu_{R\cdot x}=\int_{r\in R} \nu_{r\cdot x} \; dh(r) \] where $h$ is the Haar measure on $R$.

This is a well defined measurable assignment of finitely additive measures. We check that this is invariant. Note that if $R\cdot x$ is connected to $R\cdot y$, then by Proposition \ref{gtilde} we can take $y$ to be $\gamma\cdot x$ for some $\gamma\in \Gamma$. We have
\[r\,\ip{\gamma}\cdot x=\ip{r(\gamma\inv)}\inv \ip{r(\gamma\inv)}\,r\,\ip{\gamma} \cdot x\] and $\ip{r(\gamma\inv)}\,r\gamma\in R$. The map $r\mapsto \ip{r(\gamma\inv)}\,r\,\ip{\gamma}$ is a measure preserving bijection with inverse $r'\mapsto \ip{r'(\gamma)}\,r'\ip{\gamma\inv}$. So, if we let $R_\delta=\{r\in R:r(\gamma\inv)=\delta\}$, then $R=\bigsqcup_{\delta} R_\delta=\bigsqcup_{\delta} \delta R_\delta \gamma$. So, using the $\Gamma$-invariance of $\nu$ and unimodularity of $R$ we have
\begin{align*}
\int_R \nu_{r\,\ip{\gamma}\cdot x}dh &= \sum_\delta \int_{R_\delta} \nu_{\delta\inv\delta r \gamma \cdot x} \; dh \\
\; & = \sum_\delta \int_{R_\delta} \nu_{\delta r \gamma}\; dh \\
\; & = \sum_\delta \int_{\delta R_\delta \gamma} \nu_{r\cdot x}\; c(\delta\cdot h\cdot \gamma)   \\
\; & = \int_R \nu_{r\cdot x}\; dh.
\end{align*}

So, $\widetilde{\s S}(G)$ has an amenable connectedness relation is hyperfinite. That is, we have an increasing sequence of measurable subgraphs $E_n\subseteq \widetilde{\s S}(G)$ with finite components so that, for all $i$,
\[\bigcup_n\{x: B_i(x)\subseteq [x]_{E_n}\}=\widetilde{ \s S}(\Gamma, E).\] For any $i\in\N$, there is some $n$ with $B_i(x)$ contained in the $E_n$-component of $x$ for all but $\epsilon$-many $x$. By Lusin--Novikov we can choose a partial Cayley diagram on each $E_n$-component to get a Cayley diagram off a set of measure $\epsilon$.

\end{proof}

If we aren't allowed even an $\epsilon$ of error, the measurable combinatorics of amenable groups can be very intricate. For a survey of the situation on grids (i.e. on Cayley graphs of $\Z^n$), see \cite{gridsSurvey}. 

The best tool we have for ruling out factor of iid processes in the amenable setting (and often in the non-amenable setting) is ergodicity. For Cayley diagrams, we have the following criterion:

\begin{prop}\label{subgp}
Suppose there is a measurable subgroup $H<\aut(G)$ where $\Gamma\subseteq H$ and \[1<[\aut_e(G):H\cap \aut_e(G)]<\infty.\] Then $G$ does not admit an $\aut(G)$-fiid $\Gamma$-Cayley diagram.
\end{prop}
\begin{proof}

Since $\Gamma<H$, $\aut_e(G)\cap H$ is invariant under the action of $\Gamma$ on $\aut_e(G)$. And, $H$ has Haar measure $1/[\aut_e(G):H\cap \aut_e(G)]$, so the action of $\Gamma$ on $(\aut_e(G), h)$ is not ergodic. By Proposition \ref{prop:actionrotationdiagram} neither is the action of $\Gamma$ on $(\cd,\mu)$. But since the shift action of $\Gamma$ is ergodic, every $\Gamma$-\fiid measure is ergodic.
\end{proof}

For grids, and in fact for Cayley graphs of any torsion-free nilpotent group, it turns out that there is separation between the $\aut(G)$- and $\Gamma$-fiid labellings.

\begin{thm} \label{nilpotent}
If $\Gamma$ is torsion-free and nilpotent and $E$ is any set of generators which does not induce a GRR of $\Gamma$, then $\cay(\Gamma,E)$ does not admit an $\aut(G)$-fiid Cayley diagram.
\end{thm}
\begin{proof}
By \cite[Theorem 1.2]{nilpotent}, $\aut_e(G)$ is the set of group automorphism which preserve $E$. In particular, $\aut_e(G)$ is finite. Thus $\Gamma\cap \aut_e(G)=\{e\}$ is finite index in $\aut_e(G)$ and the previous proposition applies with $H=\Gamma$.
\end{proof}

Indeed, this proposition applies to any $CI_f$ group (in the sense of \cite{CI}). The next theorem shows that the assumption that $\Gamma$ is torsion-free is necessary for the corollary above. Consider the discrete Heisenberg group 
\[H_3(\Z):=\left\{\begin{pmatrix} 1 & a & b \\ 0 & 1 & c \\ 0 & 0 & 1 \end{pmatrix}: a,b,c\in \Z\right\}.\] This is a torsion-free nilpotent group generated by the matrices
\[A = \begin{pmatrix} 1 & 1 & 0 \\ 0 & 1 & 0 \\ 0 & 0 & 1 \end{pmatrix} \quad B = \begin{pmatrix} 1 & 0 & 0 \\ 0 & 1 & 1 \\ 0 & 0 & 1 \end{pmatrix}\]

\begin{thm} \label{nilpex}
Let $F=\{A, B, A^2, A^2B, B^2\}\subseteq H_3(\Z)$, $\Gamma=C_2\times H_3(\Z)$, and $E=C_2\times (F\cup F\inv)$. Then $G=\cay(\Gamma, E)$ admits an $\aut(G)$-\fiid Cayley diagram and $|CD|=|\aut_e(G)|>1$.
\end{thm}
\begin{proof}
Let $G'=\cay(H_3(\Z), F).$ By inspection of the unit ball in $G'$, this is a GRR of $H_3(\Z).$ The Cayley graph $G$ is obtained from $G'$ by replacing each edge in $G'$ with a copy $K_4.$


If $r\in \aut_e(G)$ is an automorphism of $G$, then by \cite[Theorem 1.11]{nilpotent} it descends to an automorphism of $G'$, which must be trivial. So $r$ decides independently for each $m\in H_3(\Z)$ whether or not to swap $(0,m)$ and $(1,m)$.

In terms of Cayley diagrams, the unique $\aut_e(G)$-invariant random $\bf c\in\cd$ can be gotten as follows: let $\bf x$ be an iid coin flip for each $m\in H_3(\Z)$ and set $\bf c((i,m),(j,n))=(i+j, mn\inv)$ if $\bf x(m)=\bf x(n)$ or $(1+i+j, mn\inv)$ otherwise. This is clearly fiid.

\end{proof}


The combinatorics of nonamenable groups are much wilder. One can write down artificial free products and amalgamations which do or don't admit (approximate) Cayley diagrams. We will show that trees groups do admit approximate $\aut(T_n)$-Cayley diagrams (generalizing a theorem of Lyons and Nazarov), and give an example of a marked group whose associated Bernoulli graphing has strictly larger approximate chromatic number than its Bernoulli shift graph. In fact, our construction answers a question of Weilacher. 

\begin{thm}
    If $\cay(\Gamma)=T_n$, then $T_n$ admits an $\aut(T_n)$-fiid approximate $\Gamma$-Cayley diagram.
\end{thm}
\begin{proof}
 Recall that if the Cayley graph of $\Gamma$ is a tree, then $\Gamma=C_2^{*n}*F_m$ for some $n,m\in\N$. We proceed by induction on $n$. We may assume we have $n+1$ independent labels on each vertex of $G$.

If $n=0$, then we want a Cayley diagram for a free group. If $m=1$, then $\Gamma$ is amenable, and we can apply Theorem \ref{approxamen}. Otherwise let $o$ be a balanced orientation of $\widetilde{ \s S}(\Gamma, E)$. One exists by \cite[Theorem 2.8]{me}. We want to label edges so that, away from a set of measure $\epsilon$, every vertex has one in-edge and one out-edge with each label. This is equivalent to approximately edge $m$-coloring the graph \[H=(\Fr([0,1]^{\Gamma})/R\times \{0,1\}, \{\{(x,0),(y,1)\}: (x,y)\in o\}).\] Note that $H$ is an acyclic and $m$-regular pmp graph. By Hatami--Lovasz--Szegedy, $H$ local-global contains $\widetilde{\s S}(T_m)$, and by Lyons--Nazarov, $T_m$ has an $\aut(T_m)$-fiid approximate edge $m$-coloring. So, $H$ has an appoximate edge $m$-coloring as desired.  

If $m=0$ and $n=1$ or $2$ then again $\Gamma$ is amenable, and so we can apply Theorem \ref{approxamen}.

Now suppose that $n+2m>3$. Then $\Gamma$ is nonamenable, so by the Lyons--Nazarov matching theorem \cite[Theorem 2.4]{LN} we can use the $n^{th}$ independent variable on each vertex to produce an $\aut(G)$-\fiid perfect matching $M$. Throwing away the edges in $M$ reduces us to looking at the $(n-1)$ case, and we're done by induction.    
\end{proof}



The difference between $\aut(G)$- and $\Gamma$-\fiid combinatorics can be reflected in vertex colorings. Recall that the approximate chromatic number of a marked group $(\Gamma, E)$ is the minimal $n$ so that $\s S(\Gamma, E)$ admits an approximate $n$-coloring. We will construct marked groups with isomorphic Cayley graphs but different approximate chromatic numbers, answering \cite[Problem 2]{felix}. 

Consider the groups \[\Gamma=C_2^{*3}\times C_3=\ip{a_1,a_2,a_3,\bx: a_i^2=[a_i,b]=b^3=1}\]
\[E = \{a_1,a_2,a_3,b\}\] and \[\Delta= \ip{a_1,a_2,a_3, \bx: a_i^2=b^3=abab=e},\]
\[F=\{a_1,a_2,a_3,b\}.\] Note that $\Delta$ is the semidirect product $C_2^{*3}\rtimes C_3$, where $C_2$ acts on $C_3$ by inversion, and $\cay(\Gamma, E)\cong \cay(\Delta, F)$.

\begin{thm}\label{3coloring}
$\cay(\Gamma, E)$ admits a $\Gamma$-\fiid 3-coloring.
\end{thm}
\begin{proof}
As in \cite{felix} it suffices to find a measurable independent set $A$ which intersects each $C_3$ orbit in $[0,1]^\Gamma$. Then we can have a 3-coloring of the shift graph given by $c(v)={\min\{i: b^i\cdot x\in A\}}$. We construct such an $A$ by an augmenting chain argument.

Any independent set meets each $C_3$ orbit at most once. And if an indepdendent set cannot be extended to meet a $C_3$ orbit then is must meet each of the three adjacent orbits. If we have any two such missing orbits in the same component, we can augment $A$ along the path between them as follows.

Given an independent set $A$, say that $x,v_1,...,v_n,y$ is an augmenting chain of length $n$ for $A$ if $x$ and $v_1$ are neighbours, $v_n$ and $y$ are neighbours, $x,y\not\in A$, $v_i\in A$, and $C_3\cdot x,C_3\cdot v_1,..., C_3\cdot y $ is a simple path in $G/C_3.$

We show that if $p=(x,v_1,...,v_n,y)$ is an augmenting chain for $A$, then there is an independent set $A'$ such that $A'\triangle A\subseteq C_3\cdot p$, and $A'$ meets the $C_3$ orbit of $x$ and every $v_i$.

 If $x$ has a neighbour outside of $C_3\cdot p$ as well, then by pigeonhole there is some point of $C_3\cdot x$ with no neighbour in $A$, say $x'$. Set $A'=A\cup \{x'\}$. Otherwise, for each $i$, $C_3\cdot v_i$ contains two points not adjacent to elements $A$ outside of $p$, say $v_i, v_i'$. Put $x$ in $A'$ and swap $v_1$ for $v_1'$. If $v_1'$ is adjacent to $v_2$, swap $v_2$ for $v_2'$. Continue in this manner until we we no longer need to swap. This process must terminate when we reach $y$.

Now as usual, we iteratively augment along all chains of length at most $n$ to produce larger and larger independent sets $A_n$. Since the $C_3$ orbits which avoid $A_n$ are spaced at least $n$ vertices apart, the set of vertices whose orbits miss $A_n$ decays exponentionally in measure. A Borel-Cantelli argument then shows that this process converges almost everywhere (cf the proof of \cite[Theorem 2.4]{LN} or \cite[Theorem 2.4]{me}).

\end{proof}

\begin{thm}
$\cay(\Delta, F)$ does not admit a $\Delta$-\fiid approximate 3-coloring.
\end{thm}
\begin{proof}
 Suppose $f$ is a measurable vertex labelling of the shift graph which meets the 3-coloring constraint on the $C_3$-orbit of $x$ and all of its neighbours. Again following \cite{felix}, let $c_f(x)=f(b\cdot x)-f(x)\in C_3$. Then
 \begin{enumerate}
     \item $c_f$ takes values in $\{\pm 1\}$, otherwise $f(b\cdot x)=f(x)$
     \item $c_f$ is constant on $C_3$ orbits, otherwise we have $f(b^2\cdot x)=f(x)+1-1=f(x)$ for some $x$.
     \item $c_f$ is different on adjacent orbits, otherwise for some $j$, $f(b^i\cdot x)=f(x)+i$, and $f(a_jb^i\cdot x)=f(b^{-i}a_j\cdot x)=f(a_j\cdot x)-i$, which is a contradiction since $m+i=n-i$ always has a solution in $C_3$.
     \end{enumerate}
 
 Thus an approximate $3$-coloring of $\s S(\Delta, F)$ yields an approximate $2$-coloring of $\s F(C_2^{*3},\{a_1,a_2,a_3\} )$ as follows. Define \[g:([0,1]^3)^{C_2^{*3}}\rightarrow [0,1]^\Delta\] by
 \[g(x_1,x_2,x_3)(\gamma b^i)=x_i(\gamma)\] for $\gamma\in C_2^{*3}$. Then if $\ip{f_n:n\in\N}$ is an approximate 3-coloring of $\s S(\Delta, F)$, we have an approximate 2-coloring of $\s F(C_2^{*3},\{a_1,a_2,a_3\} )$ given by $\ip{c_{f_n}\circ g:n\in\N}$. But then each color set is approximately invariant under the group elements of even length, contradicting strong ergodicity.
\end{proof}

\begin{cor}
$\cay(\Gamma, E)$ does not admit an approximate $\aut(\cay(\Gamma, E))$-\fiid Cayley diagram.
\end{cor}

\section{Open questions}

We'll end with a handful of open questions. On the topic of graph limit, the biggest open problem is to characterize the set of limits of finite graphs. This can be thought of as a instance of a general soficity problem, and a solution does not seem to be in sight. Even the following appears to be open:

\begin{prb}
    Is there a regular treeing which is not a local-global limit of finite graphs?
\end{prb}

All the applications of local-global convergence in the literature involve $\Sigma_1$-definable properties. Ostensibly, the equivalence of elementary and $\Sigma_1$-convergence gives us more leverage. This leads to an open ended questions:

\begin{prb}
    Do any $\Sigma_n$ properties of atomless pmp graphs yield interesting finitary theorems?
\end{prb}

On the topic Cayley diagrams, all the exact fiid Cayley diagrams we can construct involve making iid choices about generators from a finite subgroup. Is this a necessary features of fiid Cayley diagrams?

\begin{prb}
    Is there a torsion free marked group $\Gamma$ so that $\cay(\Gamma)$ admits an $\aut(\cay(\Gamma))$-fiid Cayley diagram?
\end{prb}

For nonamenable groups, it is generaly a difficult problem to separate the combinatorics on measure $(1-\epsilon)$ sets from combinatorics on measure $1$ sets. Tools like entropy estimates and strong ergodicity typically give numerically robust theorems. 

\begin{prb}
    Is there a nonamenable group $\Gamma$ so that $\cay(\Gamma)$ admits an approximate $\aut(\cay(\Gamma))$-fiid Cayley diagram, but no $\aut(\cay(\Gamma))$-fiid Cayley diagram?
\end{prb}

The answer to at least one of these must yes as witnessed by the $2n$-regular tree and the free group $F_n$. The following is still frustratingly open:

\begin{prb}
    Does the $d$-regular tree admit an $\aut(T_{d})$-fiid Cayley diagram for any marked group $\Gamma$?
\end{prb}

\bibliographystyle{plain}
\bibliography{refs}

\begin{thebibliography}{10}

\bibitem{msrthry}
H.~Bercovici, A.~Brown, and C.~Pearcy.
\newblock {\em Measure and Integration}.
\newblock Springer International Publishing, 2016.

\bibitem{BR}
Béla Bollobás and Oliver Riordan.
\newblock Sparse graphs: Metrics and random models.
\newblock {\em Random Structures \& Algorithms}, 39, 08 2011.

\bibitem{BurtonKechrisSurvey}
Peter~J. Burton and Alexander~S. Kechris.
\newblock Weak containment of measure-preserving group actions.
\newblock {\em Ergodic Theory and Dynamical Systems}, 40(10):2681–2733, 2020.

\bibitem{CKTD}
Clinton~T. Conley, Alexander~S. Kechris, and Robin~D. Tucker-Drob.
\newblock Ultraproducts of measure preserving actions and graph combinatorics.
\newblock {\em Ergodic Theory and Dynamical Systems}, 33(2):334–374, 2013.

\bibitem{gridsSurvey}
Jan Grebík and Václav Rozhoň.
\newblock Local problems on grids from the perspective of distributed algorithms, finitary factors, and descriptive combinatorics.
\newblock {\em Advances in Mathematics}, 431:109241, 2023.

\bibitem{HLS}
Hamed Hatami, L\'aszl\'o Lov\'as, and Bal\'asz Szegedy.
\newblock Limits of locally-globally convergent graph sequenes.
\newblock {\em Geometric and Functional Analysis}, 24:269--296, 2 2014.

\bibitem{Kechris}
A.~Kechris.
\newblock {\em Classical Descriptive Set Theory}.
\newblock Graduate Texts in Mathematics. Springer New York, 2012.

\bibitem{GlobalAspects}
Alexander Kechris.
\newblock {\em Global Aspects of Ergodic Group Actions}, volume 160.
\newblock AMS Mathematical Surveys and Monographs, 2010.

\bibitem{KM}
Alexander Kechris and Benjamin Miller.
\newblock {\em Topics in orbit equivalence}, volume 1852.
\newblock Springer LNM, 2004.

\bibitem{Lyons2016FactorsOI}
Russell Lyons.
\newblock Factors of iid on trees.
\newblock {\em Combinatorics, Probability and Computing}, 26(2):285–300, 2017.

\bibitem{LN}
Russell Lyons and Fedor Nazarov.
\newblock Perfect matchings as iid factors on non-amenable groups.
\newblock {\em European Journal of Combinatorics}, 32(7):1115--1125, 2011.

\bibitem{nilpotent}
David~Witte Morris, Joy Morris, and Gabriel Verret.
\newblock Isomorphisms of {Cayley} graphs on nilpotent groups.
\newblock {\em New York Journal of Mathematics}, 22:453--467, 2016.

\bibitem{CI}
Joy Morris.
\newblock The {CI} problem for infinite groups.
\newblock {\em The Electornic Journal of Combinatorics}, 23, 12 2016.

\bibitem{RahmanVirag}
Mustazee Rahman and B{\'a}lint Vir{\'a}g.
\newblock {Local algorithms for independent sets are half-optimal}.
\newblock {\em The Annals of Probability}, 45(3):1543 -- 1577, 2017.

\bibitem{me}
Riley Thornton.
\newblock Orienting borel graphs.
\newblock {\em Proceedings of the American Mathematical Society}, 150, 01 2020.

\bibitem{felix}
Felix Weilacher.
\newblock Marked groups with isomorphic cayley graphs but different borel combinatorics.
\newblock {\em Fundamenta Mathematicae}, 2018.

\bibitem{Survey}
Itaï~Ben Yaacov, Alexander Berenstein, C.~Ward Henson, and Alexander Usvyatsov.
\newblock {\em Model theory for metric structures}, page 315–427.
\newblock London Mathematical Society Lecture Note Series. Cambridge University Press, 2008.

\end{thebibliography}

\end{document}